\documentclass{siamltex}

\usepackage{thumbpdf}
\usepackage[%
  colorlinks%
  ,bookmarks={false}%
  ,pdftitle={}%
  ,pdfsubject={}%
  ,pdfkeywords={}%
  ,pdfauthor={Tom Maerz <maerz@maths.ox.ac.uk>}%
  ]{hyperref}

\usepackage{a4wide}
\usepackage{amsmath}
\usepackage{amssymb}
\usepackage{mathpazo}
\usepackage{graphicx}
\usepackage{enumerate}
\usepackage[latin1]{inputenc}

\newtheorem{theo}{Theorem}[section]
\newtheorem{defi}[theo]{Definition}
\newtheorem{prin}[theo]{Principle}

\newtheorem{lem}[theo]{Lemma}
\newtheorem{prop}[theo]{Proposition}

\newcommand{\R}{{\mathbb  R}}

\newcommand{\laplace}{\Delta}
\newcommand{\Hm}{{\mathcal H}}
\newcommand{\Tang}{{\mathcal T}}
\newcommand{\Norm}{{\mathcal N}}
\newcommand{\Order}{{\mathcal O}}
\newcommand{\bd}{\partial}
\newcommand{\pd}{\partial}
\newcommand{\wo}{\setminus}
\newcommand{\dotarg}{\, \cdot \,}

\DeclareMathOperator{\diver}{div}
\DeclareMathOperator{\trace}{trace}
\DeclareMathOperator{\cp}{cp}
\DeclareMathOperator{\ecp}{ecp}
\DeclareMathOperator{\sd}{sd}

\DeclareMathOperator{\id}{id}
\DeclareMathOperator{\codim}{codim}
\DeclareMathOperator{\spa}{span}

\title{Calculus on Surfaces with General Closest Point Functions}
\author{Thomas M\"{a}rz, Colin B.~Macdonald
	\thanks{ Mathematical Institute,
				University of Oxford,
				OX1 3LB, UK.
				Email: {\tt \{maerz,macdonald\}@maths.ox.ac.uk}.
				This work was supported by award KUK-C1-013-04 made by King Abdullah University of Science and Technology (KAUST).
				Manuscript as of \today.}}

\begin{document}

\maketitle

\begin{abstract}
  The Closest Point Method for solving partial differential equations
  (PDEs) posed on surfaces was recently introduced by Ruuth
  and Merriman [J. Comput. Phys. 2008] and successfully applied to a
  variety of surface PDEs.
  In this paper we study the theoretical foundations of this method.
  The main idea is that surface differentials of a surface function can be
  replaced with Cartesian differentials of its closest point extension, i.e.,
  its composition with a closest point function.
  We introduce a general class of these closest point functions (a subset of differentiable retractions),
  show that these are exactly the functions necessary to satisfy the above idea,
  and give a geometric characterization of this class.
  Finally, we construct some closest point functions and demonstrate
  their effectiveness numerically on surface PDEs.
\end{abstract}



\pagestyle{myheadings}
\thispagestyle{plain}
\markboth{T. M\"{A}RZ, C. B. MACDONALD}{CALCULUS ON SURFACES WITH GENERAL CLOSEST POINT FUNCTIONS}

\section{Introduction}\label{sect:Intro}
The Closest Point Method is a set of mathematical principles and
associated numerical techniques for solving partial differential
equations (PDEs) posed on surfaces.
It is an embedding technique
and is based on an implicit representation of the surface $S$.
The original method \cite{sjr:SimpleEmbed} uses a representation induced by Euclidean distance;
specifically, for any point $x$ in an embedding space $\R^n$ containing $S$, a point $\ecp(x) \in S$ is
known which is closest to $x$ (hence the term ``Closest Point Method''). The function $\ecp$ is the Euclidean closest point function.

This kind of representation of a surface $S$ will be generalized: we introduce a general class of \emph{closest point functions} (and denote a member by ``$\cp$'')
within the set of
differentiable retractions. The \emph{closest point extension} of a surface function $u$ is given by $u(\cp(x))$.
The common feature of all closest point functions $\cp$ is that the extension $u(\cp(x))$ locally propagates data $u$ perpendicularly off the surface $S$ into the surrounding space $\R^n$.
In this way, closest point extensions lead to simplified derivative calculations in
the embedding space---because $u(\cp(x))$ does not vary in the
direction normal to the surface.
More specifically, we have (for every closest point function):

\emph{Gradient Principle}: intrinsic gradients of surface functions
agree on the surface with Cartesian gradients of the closest point
extension.

\emph{Divergence Principle}: surface divergence operators of surface
vector fields agree on the surface with the Cartesian divergence
operator applied to the closest point extension of those vector
fields.

These are stated more precisely and proven as Principles~\ref{cor:PrinGrad}
and~\ref{cor:PrinDiv}.  Combinations of these two principles may be
made, to encompass higher order differential operators for example the
Laplace--Beltrami and surface biharmonic operators.
These principles can then be used to replace the intrinsic spatial
operators of surface PDEs with Cartesian derivatives in the embedded
space.

Numerical methods based on these principles are compelling because
they can re-use simple numerical techniques on Cartesian grids such as
finite difference methods and standard interpolation schemes
\cite{sjr:SimpleEmbed}.  Other advantages include the wide variety of
geometry that can be represented, including both open and closed
surfaces with or without orientation in general codimension (e.g.,
filaments in 3D \cite{sjr:SimpleEmbed} or a Klein bottle in 4D
\cite{cbm:lscpm}).  In this way, the Closest Point Method has been
successfully applied to a variety of time-dependent problems including
in-surface advection, diffusion, reaction-diffusion, and
Hamilton--Jacobi equations \cite{sjr:SimpleEmbed, cbm:lscpm},
where standard time integration schemes are used.  It has been shown
to achieve high-order accuracy \cite{cbm:lscpm,cbm:icpm}.  It has also been used for
time-independent problems such as eigenvalue problems on surfaces
\cite{cbm:eigenmodes}.

The remainder of this paper unfolds as follows. In
Section~\ref{sect:calc}, we review both notation and a calculus on
embedded surfaces which does not make use of parametrizations.  In
Section~\ref{sect:calcCP}, we define our class of closest point
functions within the class of differentiable retractions.  The key
property is that, for points on the surface, the Jacobian matrix of
the closest point function is the projection matrix onto the tangent
space.  We show this class of functions gives the desired simplified
derivative evaluations and that it is the largest class of retractions that does so.
We thus prove all of these functions induce the closest point
principles of \cite{sjr:SimpleEmbed} (in fact, a more general
divergence operator is established).  Finally, we give a geometric
characterization of these functions (namely that the pre-image of the
closest point function intersects the surface orthogonally).
Section~\ref{sect:diffusion} discusses general diffusion operators:
these can be treated with the principles established in
Section~\ref{sect:calcCP} by using two extensions but there are also
many interesting cases (including the Laplace--Beltrami operator)
where one can simply use a single extension.
Section~\ref{sect:Const} describes one possible construction method
for non-Euclidean closest point functions which makes use of a
multiple level-set description of $S$. This construction method can be
realized numerically which we demonstrate in an example.  Finally, in
Section~\ref{sect:CPM} we use the non-Euclidean closest point
representation to solve an advection problem and a diffusion problem
on a curve embedded in $\R^3$.


\section{Calculus on Surfaces without Parametrizations}\label{sect:calc}
In this section, we review notation and definitions to form a calculus on surfaces embedded in $\R^n$
(see e.g., \cite{GLKB:Calc, JWM:TopoDiffView, JO:DiffGeoApps, AS:MCMAC, DDEH:hNarrow, DE:EulerianFEM, DG:Beltrami}).

\subsection{Smooth Surfaces}\label{subsect:surf}
Throughout this paper we consider smooth surfaces $S$ of dimension $\dim S = k$ embedded in $\R^n$ ($k \leq n$) which possess a tubular neighborhood \cite{MWH:DiffTopo}.
We refer to this tubular neighborhood as $B(S)$, a \emph{band} around $S$.

In the case that $S$ has a boundary $\bd S \neq \emptyset$, we identify $S$ with its interior $S = S \wo \bd S$.
Moreover, we assume that $S$ is orientable, i.e., if $\codim S = n-k \neq 0$ there are smooth vector fields
$N_i : S \to \R^n$, $i = 1,\ldots, n-k$
which span the normal space $\Norm_y S = \spa\{ N_1(y),\ldots, N_{n-k}(y)\}$ (see e.g., \cite[Proposition 6.5.8]{AMR:Manifolds}).
The vectors $N_i(y)$ are not required to be pairwise orthogonal, but we require them to be normed $|N_i(y)|=1$.
Finally, we define the $\R^{n \times k}$-matrix $N(y) := (N_1(y) | \ldots | N_{n-k}(y))$ which contains all the normal vectors as columns.

The tubular neighborhood assumption is sufficient for the existence of retractions (see \cite{MWH:DiffTopo}). This is important because the closest point functions
defined in Section~\ref{sect:calcCP} are retractions. Note that every smooth surface embedded in $\R^n$ without boundary has a tubular neighborhood by \cite[Theorem 5.1]{MWH:DiffTopo}.
Moreover, the orientability side condition is not restrictive since it will be satisfied locally when referring to sufficiently small subsets of the surface.

The matrix $P(y)$ denotes the orthogonal projector that projects onto the tangent space $\Tang_y S$ of $S$ at $y \in S$.
$P$ as a function of $y \in S$ is a tensor field on the surface and can be written in terms of the normal vectors as
\[
	P:S \to \R^{n \times n} \;, \quad P(y) = I - N(y) \cdot N(y)^{\dagger}
\]
where $N(y)^{\dagger}$ denotes the pseudo-inverse of $N(y)$.
If $\codim S = 1$ then $P$ is given by $P(y) = I - N(y) \cdot N(y)^T$ since the matrix $N(y)$ has only a single normed column vector.

Regarding the degree of smoothness: later when we speak about $C^l$-smooth surface functions and differentials up to order $l$,
we assume that the underlying surface $S$ is at least \mbox{$C^{l+1}$-smooth}.

\subsection{Smooth Surface Functions and Smooth Extensions}\label{subsect:FuncAndExt}
Given a surface $S$, we consider smooth surface functions: a scalar function $u:S \to \R$, vector-valued function $f:S \to \R^m$, and vector field $g:S \to \R^n$.
We call $g$ a vector field because it maps to $\R^n$, the embedding space of $S$, and hence we can define a divergence operation for it.
The calculus without parametrizations for such surface functions is based on smooth extensions, defined in the following.

\begin{defi}{(Extensions)}\label{def:ext}
	We call $u_E: \Omega_E \to \R$ an extension of the surface function $u:S \to \R$---and likewise for vector-valued surface functions---if the following properties hold:
	\begin{itemize}
		\item $\Omega_E \subset \R^n$ is an open subset of the embedding space.
		\item $\Omega_E$ contains a surface patch $S \cap \Omega_E \neq \emptyset$.
		\item $u_E|_{S \cap \Omega_E} = u|_{S \cap \Omega_E}$, the extension and the original function coincide on the surface patch.
		\item $u \in C^l(S \cap \Omega_E) \, \Rightarrow \, u_E \in C^l(\Omega_E)$, the extension is as smooth as the original function.
	\end{itemize}
	As extensions are not unique, $u_E$ denotes an arbitrary representative of this equivalence class.
\end{defi}

Different extensions might have different domains of definition.
Here $\Omega_E$ is just a generic name for an extension domain that suits the chosen extension and that contains the surface point $y \in S$ which is under consideration.

An important point is that such extensions exist and may be obtained by using the following feature of a smooth surface:
for every $y\in S$ there are open subsets $\Omega_E$,$W$ of $\R^n$, where $y \in \Omega_E$, and a diffeomorphism
$\psi : \Omega_E \to W$ ($\psi, \psi^{-1}$ are as smooth as $S$) that locally flattens $S$
\[
	\psi( \Omega_E \cap S ) = W \cap \R_0^k \;,\quad k = \dim S \;, \quad \R_0^k  := \left\{ x \in \R^n : x=(x_1,\ldots,x_k,0,\ldots,0) \right\} \;,
\]
see e.g., \cite[Chapter 3.5]{KK:Ana2}.
Now, $\psi^{-1}: W \cap \R_0^k \to \Omega_E \cap S$ parametrizes the patch $\Omega_E \cap S$ of $S$, so $u \circ \psi^{-1} : W \cap \R_0^k \to \R$
is a smooth function. Let $\tilde{P} \in \R^{n \times n}$ be the projector onto $\R_0^k$.
For $x \in \Omega_E$, the function $\tilde{P} \cdot \psi(x)$ maps  smoothly onto $W \cap \R_0^k$, hence
\[
	u_E : \Omega_E \to \R \;,\quad u_E(x) := u \circ \psi^{-1} \left( \tilde{P} \cdot \psi(x) \right)
\]
is smooth and extends $u$ as desired.

\subsection{Calculus without Parametrizations}
Now, we define the basic first order differential operators on $S$ without the use of parametrizations
(compare to e.g., \cite{GLKB:Calc, JWM:TopoDiffView, JO:DiffGeoApps, AS:MCMAC, DDEH:hNarrow, DE:EulerianFEM, DG:Beltrami}):

\begin{defi}\label{def:diffS}
    Let $S$ be a smooth surface embedded in $\R^n$, $y \in S$ a point on the surface, and $P(y) \in \R^{n \times n}$ the projector (onto $\Tang_y S$) at this point $y$.
    Then the surface gradient $\nabla_S$ of a scalar $C^1$-function $u$, the surface Jacobian $D_S$ of a vector-valued $C^1$-function $f$, and the surface divergence $\diver_S$ of a $C^1$ vector field $g$
    are given by:
    \begin{align*}
	    \nabla_S u(y)^T &:= \nabla u_E(y)^T \cdot P(y) \;,\\
	    D_S f(y) &:= D f_E(y) \cdot P(y) \;,\\
	    \diver_S g(y) &:= \trace( D_Sg(y) )  = \trace( D g_E(y) \cdot P(y) ) \;.
    \end{align*}
    Here, $\nabla$ is the gradient and $D$ the Jacobian in the embedding space $\R^n$ applied to the extensions of the surface functions.
	 The extensions are arbitrary representatives of the equivalence classes of Definition \ref{def:ext}.
\end{defi}

\paragraph{Remark}
There are other works (see e.g., \cite{ABL:HypCLawsOnMan}) that use a variational definition of
the surface divergence operator, which we denote by $\diver_S^* g$, as
\[
	\int_{\Omega} v \; \diver_S^* g \; d\Hm^n(y) = -\int_{\Omega} \left< \nabla_S v, g \right > \; d\Hm^n(y)
	\;,\quad \forall \; v \in C_0^{1}(\Omega) \;, \; \Omega \subset S \;.
\]
But this definition applies only to \emph{tangential} vector fields $g$, because this tangency is required by the surface Gauss--Green Theorem.
The connection to our definition is as follows
\[
	\diver_S^* g = \diver_S (Pg) \;,
\]
because $\diver_S^*$ takes into account only the tangential part $Pg$ of the vector field $g$.
The two are equal if the vector field $g$ is indeed tangential to the surface.
The trace-based definition of surface divergence in Definition~\ref{def:cpop} also applies to non-tangential fields \cite{AS:MCMAC}.
For example when $\codim S = 1$ and $g$ is a non-tangential vector field, it can be shown that
\[
	\diver_S g = \diver_S (Pg) + \diver_S ( \left< N, g \right> \cdot N) = \diver_S (Pg) - \left< N, g \right> \kappa_S \;,
\]
where $N$ is the normal vector, and $\kappa_S= -\diver_S N$ is the mean curvature of $S$.
That is, the surface divergence of $g$ is the surface divergence of the tangential component plus an extra term which depends on the curvature(s) of $S$.

\section{Calculus on Surfaces with Closest Point Functions}\label{sect:calcCP}
Our aim is to compute surface intrinsic derivatives by means of \emph{closest point functions}.
For now, we consider only first order derivatives, higher order derivatives are discussed in Section~\ref{subsect:DiffHOrder}.
By the remark on smoothness in Section~\ref{subsect:surf}, this means we are considering $C^2$-smooth surfaces.
Closest point functions are retractions with the key property that their Jacobian evaluated on the surface is the projector $P$.

\begin{defi}{(Closest Point Functions)}\label{def:cpop}
	We call a map $B(S) \to S$ a closest point function $\cp$, if
	\begin{enumerate}[1.]
		\item $\cp$ is a $C^1$-\emph{retraction}, i.e., $\cp: B(S) \to S$ features the properties
		\begin{enumerate}[a)]
			\item $\cp \circ \cp = \cp$ or equivalently $\cp|_S = \id_S$,\footnote{Here $\id_S:S \to S$ with $\id_S(y)=y$ denotes the identity map on the surface $S$.}
			\item $\cp$ is continuously differentiable.
		\end{enumerate}
		\item $D \cp (y) = P(y)$ for all $y \in S$.
	\end{enumerate}
\end{defi}

If $\cp$ belongs to this class of closest point functions, we recover surface differential operators (those in Definition~\ref{def:diffS})
by standard Cartesian differential operators applied to the closest point extensions $u \circ \cp$.
This fundamental point is the basis for the Closest Point Method and is established by the following theorem:

\begin{theo}{(Closest Point Theorem)}\label{theo:cpopDiff}
	If $\cp: B(S) \to S$ is a closest point function according to Definition~\ref{def:cpop}, then the following rule
	\begin{equation}\label{eqn:cpCalc}
		  D [f \circ \cp](y) = D_S f(y) \;,  \quad  y \in S,
	\end{equation}
	holds for the surface Jacobian $D_S f$ of a continuously differentiable vector-valued surface function $f:S \to \R^m$.
\end{theo}

We also show that any $C^1$-retraction that satisfies the relation in \eqref{eqn:cpCalc} must be a closest point function satisfying Definition~\ref{def:cpop}.

\begin{theo}\label{theo:cpopDiffInv}
	Let $r: B(S) \to S$ be a $C^1$-retraction (i.e., only part 1 of Definition~\ref{def:cpop}  is required).
	If for every continuously differentiable surface function $f:S \to \R^m$ the following holds
	\begin{equation}\label{eqn:cpCalc2}
		  D [f \circ r](y) = D_S f(y) \;,  \quad  y \in S,
	\end{equation}
	then $r$ satisfies also part 2 of Definition  \ref{def:cpop}
	\[
		D r (y) = P(y) \quad \text{for all} \quad y \in S,
	\]
	and is thus a closest point function.
\end{theo}

\begin{proof}
	Ad Theorem~\ref{theo:cpopDiff}: the closest point extension can be written in terms of an arbitrary extension $f_E$:
	\begin{align*}
		f \circ \cp(x) &= f_E \circ \cp(x) \;, \quad  x \in B(S) \;.
	\end{align*}
	Now, we expand the differential using the chain rule\footnote{In this paper, we follow the
	convention that differential operators occur before composition in the order of operations.
	For example, $D f_E \circ \cp(x)$ means differentiate $f_E$ and then compose with $\cp(x)$.}
	\begin{align*}
		D [f \circ \cp](x) &= Df_E  \circ \cp(x) \cdot D\cp(x) \;, \quad  x \in B(S) \;.
	\end{align*}
	Next we set $x=y \in S$, a point on the surface,
	and by Definition~\ref{def:cpop} we have $\cp(y)=y$ and $D \cp(y) = P(y)$.  Finally, using the surface differential in Definition~\ref{def:diffS} gives
	\begin{align*}
		D [f \circ \cp](y) &= Df_E(y) \cdot P(y) = D_S f(y) \;, \quad  y \in S \;.
	\end{align*}

	Ad Theorem \ref{theo:cpopDiffInv}:
	we consider the identity map on $S$, $\id_S : S \to S$, $\id_S(y) = y$. The simplest extension of $\id_S$ is given by the identity map on the embedding space, $\id: \R^n \to \R^n$.
	By Definition \ref{def:diffS}, the surface Jacobian of the surface identity is
	\begin{equation}\label{eqn:PisDid}
		D_S \id_S(y) = D\id(y) \cdot P(y) = P(y) \;, \quad  y \in S \;.
	\end{equation}
	By assumption, \eqref{eqn:cpCalc2} is true for every smooth surface function. And so with $f = \id_S$ and the previous result we have
	\[
		Dr(y) = D[\id_S \circ r](y) =  D_S \id_S(y) = P(y) \;, \quad  y \in S \;.
	\]
	\hfill
\end{proof}

Direct consequences of the Closest Point Theorem~\ref{theo:cpopDiff} are the gradient and the divergence principles below.

\begin{prin}{(Gradient Principle)}\label{cor:PrinGrad}
	If $\cp: B(S) \to S$ is a closest point function according to Definition~\ref{def:cpop}, then
	\begin{equation}\label{eqn:cpCalcGrad}
		\nabla [u \circ \cp](y) = \nabla_S u(y) \;,  \quad  y \in S,
	\end{equation}
	holds for the surface gradient $\nabla_S u$ of a continuously differentiable scalar surface function $u:S \to \R$.
\end{prin}

\begin{prin}{(Divergence Principle)}\label{cor:PrinDiv}
	If $\cp: B(S) \to S$ is a closest point function according to Definition~\ref{def:cpop}, then
	\begin{equation}\label{eqn:cpCalcDiv}
		\diver [g\circ \cp](y) = \diver_S g(y) \;,  \quad  y \in S,
	\end{equation}
	holds for the surface divergence $\diver_S g$ of a continuously differentiable surface vector field $g:S \to \R^n$.
	Notably the vector field $g$ need not be tangential to $S$.
\end{prin}

\subsection{Geometric Characterization of Closest Point Functions}
A further characteristic feature of closest point functions is that the pre-image $\cp^{-1}(y)$
of a closest point function must intersect the surface $S$ orthogonally.
This fact will be established in Theorem~\ref{theo:CPbyGeo}, for which we will need smoothness of the pre-image.

\begin{lem}\label{lem:PreimageSmooth}
	Let $r:B(S) \to S$ be a $C^1$-retraction and let $y \in S$.
	The pre-image $r^{-1}(y)$ is (locally around $y$) a $C^1$-manifold.
\end{lem}

\begin{proof}
	By differentiation of the equation $\id_S \circ r(x) = r(x)$ and setting $x=y \in S$ thereafter,  we get
	\begin{equation}\label{eqn:PDr}
		P \circ r(x) \cdot Dr(x) = Dr(x)  \qquad \Rightarrow \quad P(y) \cdot Dr(y) = Dr(y) \;.
	\end{equation}
	We see that the range of the Jacobian matrix $Dr(y)$ is contained in the tangent space $\Tang_yS$. On the other hand,
	we can differentiate
	$ r \circ \id_S(y) = \id_S(y)$ with respect to $y \in S$ and right-multiply by a matrix $T(y) \in \R^{n \times k}$ which consists of an orthonormal basis
	of $\Tang_y S$ to get
	\begin{equation}\label{eqn:DrP}
		Dr(y) \cdot P(y) = P(y)  \qquad \Rightarrow \quad  Dr(y) \cdot T(y) = T(y) \;.
	\end{equation}
	Equation~\eqref{eqn:PDr} shows that the linear operator $Dr(y): \R^n \to \R^n$ is rank deficient while $Dr(y): \Tang_y S \to \Tang_y S$ is full rank
	by \eqref{eqn:DrP} and the rank is $k = \dim S$.
	Now, we can apply the Implicit Function Theorem to
	\[
		f_y(x) = T(y)^T \cdot (r(x)-y) = 0
	\]
	where $f_y : \R^n \to \R^k$ is a $C^1$-mapping and $r^{-1}(y)$ is the set of solutions.
	The preceding discussion shows that $Df_y(y) = T(y)^T \cdot Dr(y)$ is full rank, and hence $r^{-1}(y)$ is (locally around $y$) a $C^1$-manifold.
	\hfill
\end{proof}

\begin{theo}\label{theo:CPbyGeo}
	Let $r:B(S) \to S$ be a $C^1$-retraction and let $y \in S$.
	The retraction $r$ has the property $Dr(y)=P(y)$ (and hence $r$ is a closest point function according to Definition \ref{def:cpop}),
	if and only if for every $y \in S$ the $C^1$-manifold $r^{-1}(y)$ intersects $S$ orthogonally, i.e.,
	\[
		\Tang_y r^{-1}(y) = \Norm_y S \;.
	\]
\end{theo}

\begin{proof}
	Let $y \in S$ and assume that the pre-image $r^{-1}(y)$ intersects $S$ orthogonally.
	Let $\xi: (-\varepsilon,\varepsilon) \to r^{-1}(y)$, $\varepsilon > 0$, be a regular $C^1$-curve in $r^{-1}(y)$ with $\xi(0) = y$.
	As $r^{-1}(y)$ is a $C^1$-manifold by Lemma \ref{lem:PreimageSmooth} such curves $\xi$ exist.
	Because $\xi(t) \in r^{-1}(y)$, $r$ maps every point $\xi(t)$ to $y$:
	\begin{align*}
		r \circ \xi(t) &= y \;, \quad \text{for all} \quad t \in (-\varepsilon,\varepsilon).
	\end{align*}
	Differentiation of the latter equation with respect to $t$ and the substitution of $t=0$ yield
	\begin{align}\label{eqn:DrN}
		Dr(y) \cdot \xi'(0) &= 0 \;.
	\end{align}
	By assumption $\xi'(0)$ is an arbitrary vector of the normal space $\Norm_y S$.
	Let (here) $N(y) \in \R^{n \times (n-k)}$ be a matrix where the columns form an orthonormal basis of $\Norm_y S$, then \eqref{eqn:DrN} implies
	\[
		Dr(y) \cdot N(y) = 0 \;.
	\]
	By the latter result combined with \eqref{eqn:DrP} the product of $Dr(y)$ with the $n \times n$ orthogonal matrix $(T(y) | N(y))$
	(the matrix $T(y) \in \R^{n \times k}$ consists of an orthonormal basis of $\Tang_y S$) is
	\[
		Dr(y) \cdot (T(y) | N(y)) = (T(y) | 0 \, ) \quad \Rightarrow \quad Dr(y) = T(y) \cdot T(y)^T = P(y) \;.
	\]
	This proofs the first direction.

	Let now $Dr(y)=P(y)$ for all $y \in S$, i.e., the retraction $r$ is already a closest point function.
	Let again $\xi: (-\varepsilon,\varepsilon) \to r^{-1}(y)$ be a regular $C^1$-curve in $r^{-1}(y)$ with $\xi(0) = y$.
	We consider again \eqref{eqn:DrN} but this time we replace $Dr(y)$ with $P(y)$:
	\[
		P(y) \cdot \xi'(0) = 0 \;.
	\]
	The latter tells us that every vector of $\Tang_y r^{-1}(y)$ is perpendicular to $\Tang_y S$, in other words $r^{-1}(y)$ intersects $S$ orthogonally.
	\hfill
\end{proof}

\subsection{The Euclidean Closest Point Function}
The class of closest point functions from Definition~\ref{def:cpop} is not empty. We show that the Euclidean closest point function $\ecp$
\begin{equation}
	\ecp(x) = \arg \min\limits_{y \in S} |x-y| \;, \qquad x \in B(S)
\end{equation}
---which is well-defined by our assumptions on surfaces in Section~\ref{subsect:surf}---belongs to this class.
However, this is not just a corollary of the last theorem, we have to show that $\ecp$ is continuously differentiable on the surface $S$.

\begin{theo}\label{theo:ECPisCP}
	The Euclidean closest point function $\ecp$ 
	is a closest point function satisfying Definition~\ref{def:cpop}.
\end{theo}

\begin{proof}
	The proof is instructive in the special case of $\codim S = 1$ and we begin with that case:
	The Euclidean closest point function $\ecp: B(S) \to S$ is a continuous retraction and can be written in terms of the Euclidean distance map $d$ as
	\begin{equation}\label{eqn:ecpDef}
		\ecp(x) = x - d(x) \cdot \nabla d(x) \;.
	\end{equation}
	Because $\codim S =1$, we can replace $d$ with a signed Euclidean distance $\sd$ which on $B(S)$ is a classical solution of the Eikonal equation,
	where $\sd$ is as smooth as $S$ and where $\nabla \sd(y) = N(y)$, for $y \in S$, is one of the two possible choices of the surface normal.
	So $\ecp$ is continuously differentiable with Jacobian
	\[
		D \ecp(x) = I - \nabla \sd(x) \cdot \nabla \sd(x)^T - \sd(x) \cdot D^2 \sd(x) \;,
	\]
	and, as $\sd$ vanishes on $S$, we get
	\[
		D \ecp(y) = I - N(y) \cdot N(y)^T = P(y) \; , \quad y \in S \;,
	\]
	which proves the statement in this case.

	If $S$ is of higher codimension, we cannot use this argument since $d$ is not differentiable on the surface $S$.
	But it is continuously differentiable off the surface and so is $\ecp$ by \eqref{eqn:ecpDef}. We prove now that $D \ecp$ extends continuously onto $S$ and equals the projector $P$ there.

	Let $\gamma: U \subset \R^k \to \Omega_S$, $\theta_1 \to \gamma(\theta_1)$  be a regular parametrization of
	a surface patch $\Omega_S \subset S$, ($k = \dim S$). Then, by adding the normal space
	\begin{equation}\label{eqn:InvEcp}
		\begin{aligned}
			&\Phi(\theta) = \gamma(\theta_1) + N \circ \gamma(\theta_1)  \cdot \theta_2 \;, \\
			&\\
			&\Phi: \Omega \subset \R^n \to B(S) \;, \quad \theta=(\theta_1, \theta_2) \to \Phi(\theta) \;,
		\end{aligned}
	\end{equation}
	we get a coordinate system on a corresponding subset of $B(S)$. $N(y) = (N_1(y)| \ldots | N_{n-k}(y))$ is a matrix formed of the
	normal space basis and $\theta_2 \in \R^{n-k}$. Now, we have
	\begin{equation}\label{eqn:RetEcp}
		\gamma(\theta_1) = \ecp \circ \Phi(\theta) \;.
	\end{equation}
	As long as $\theta_2 \neq 0$ (i.e., off the surface) we can differentiate \eqref{eqn:RetEcp} using the chain rule:
	\begin{equation}\label{eqn:RetEcpDiff}
		(D_{\theta_1} \gamma \; | \; 0 ) = D\ecp \circ \Phi \cdot D_{\theta} \Phi = D\ecp \circ \Phi \cdot ( D_{\theta_1} \gamma +  D_{\theta_1}(N \circ \gamma) \cdot \theta_2 \; | \; N \circ \gamma) \;.
	\end{equation}
	The second factor of the right hand side extends onto $S$ (i.e., $\theta_2=0$ is admissible) and is invertible, the left hand side is defined for $\theta_2=0$ anyway.
	After inverting, we send $\theta_2$ to zero and so we have ($\Phi(\theta_1,0) = \gamma(\theta_1)$)
	\[
		D\ecp \circ \gamma= (D_{\theta_1} \gamma \; | \; 0 ) \cdot ( D_{\theta_1} \gamma \; | \; N \circ \gamma)^{-1} \;.
	\]
	As $D_{\theta_1} \gamma$ and $N \circ \gamma$ are orthogonal, i.e., $D_{\theta_1} \gamma^T \cdot N \circ \gamma = 0$ and $N \circ \gamma^T \cdot  D_{\theta_1} \gamma = 0$,
	we can write the inverse in terms of the pseudo inverses of the sub-matrices
	\[
		( D_{\theta_1} \gamma \; | \; N \circ \gamma)^{-1} =
		\left( \begin{matrix}
					D_{\theta_1} \gamma^{\dagger} \\
					\hline
					N \circ \gamma^{\dagger}
					\end{matrix} \right)
	\]
	and so
	\[
		D\ecp \circ \gamma = (D_{\theta_1} \gamma \; | \; 0 ) \cdot ( D_{\theta_1} \gamma \; | \; N \circ \gamma)^{-1} = (D_{\theta_1} \gamma \; | \; 0 ) \cdot
		\left( \begin{matrix}
					D_{\theta_1} \gamma^{\dagger} \\
					\hline
					N \circ \gamma^{\dagger}
					\end{matrix} \right)
		= D_{\theta_1} \gamma \cdot D_{\theta_1} \gamma^{\dagger} = P \circ \gamma \;.
	\]
	In the latter equality we have used that the outer product of a full column-rank matrix $A$ and its pseudo inverse gives the projector onto the image of $A$
	which in our case is the tangent space. Finally we get the assertion:
	\[
		D\ecp \circ \gamma= P \circ \gamma \qquad \Leftrightarrow \qquad D \ecp(y) = P(y) \;, \quad y \in S \;.
	\]
	\hfill
\end{proof}

\paragraph{Remark}
As a consequence of Theorem~\ref{theo:ECPisCP} we see that the function
\[
	x-\ecp(x) = d(x) \cdot \nabla d(x) = \nabla \left( \frac{d(x)^2}{2} \right)
\]
is continuously differentiable. The article \cite{GF:VDF} introduces $v(x) := x-\ecp(x)$ as the \emph{vector distance function} and refers to the article \cite{AS:MCMAC}
for some of its features. In particular the authors of \cite{AS:MCMAC} prove that the squared distance function is differentiable and show that
the Hessian $D^2 \eta(y)$ of the function $\eta := d^2 /2$ when evaluated at a surface point $y \in S$ is exactly the projector onto
the normal space $\Norm_y S$. Here, we have an alternative proof that $\eta$ (and thereby the squared distance function) is differentiable twice
with $I-P(y) = Dv(y) = D^2 \eta(y)$ on $S$ which is the projector onto $\Norm_y S$.

\subsubsection{Additional Smoothness of the Euclidean Closest Point Function}\label{subsubsect:RemEcp}
Theorem~\ref{theo:ECPisCP} implies only $C^1$-smoothness of $\ecp$ but taking a closer look at $\Phi$ as defined in \eqref{eqn:InvEcp} we see that $\Phi$
is as smooth as $N$ (which typically is one order less smooth than the surface $S$).  Note also that $\ecp$ inherits its degree of smoothness from $N$.
For example, given a $C^3$-smooth surface $S$, we have $\Phi$ $C^2$-smooth. So we can differentiate \eqref{eqn:RetEcpDiff} and solve for $D^2 \ecp \circ\, \Phi$, again by appealing to the
non-singularity of $D \Phi$. The limit $\theta_2 \to 0$ exists here again which shows $D^2 \ecp$ has a continuous extension onto $S$, hence $\ecp$ is $C^2$-smooth for a $C^3$-smooth surface.
Applying these ideas repeatedly proves that $\ecp$ is $C^{l}$-smooth for a $C^{l+1}$-smooth surface.

\subsection{Derivatives of Higher Order}\label{subsect:DiffHOrder}
For $C^l$-smooth surface functions we can recover surface differential operators of higher order by applying Theorem~\ref{theo:cpopDiff}
(or the Gradient and Divergence Principles~\ref{cor:PrinGrad} and~\ref{cor:PrinDiv}) several times.
For example the surface Hessian $D_S^2 u$ of a scalar surface function $u \in C^l(S)$ with $l \geq 2$ can be calculated by combining Theorem~\ref{theo:cpopDiff} and the Gradient Principle~\ref{cor:PrinGrad}:
\begin{align*}
	D_S^2 u(y) = D_S \nabla_S u(y) &= D [ \nabla_S u \circ \cp ](y) = D [ \nabla [u \circ \cp] \circ \cp ](y) \;, \quad y \in S \;.
\end{align*}
Interestingly, higher order derivatives do not require more smoothness of the closest point function; $C^1$ is sufficient.
To see this expand $\nabla [u \circ \cp]$ in terms of an arbitrary extension $u_E$
\begin{equation}\label{eqn:PrinGrad}
	\nabla [u \circ \cp](x) = D\cp(x)^T \cdot \nabla u_E \circ \cp(x) \;.
\end{equation}
The clue here is that the extension of $D\cp$ is the extension of the projector $P$:
\[
	D\cp(y) = P(y) \;, \; y \in S \qquad \Rightarrow \qquad D\cp \circ \cp(x) = P \circ \cp(x),
\]
and thus a second extension of \eqref{eqn:PrinGrad} gives
\begin{align}\label{eqn:GradReext}
	\nabla [u \circ \cp] \circ \cp(x) & = P \circ \cp(x) \cdot \nabla u_E \circ \cp(x).
\end{align}
By the assumption at the end of Section~\ref{subsect:surf}, the surface must be $C^{l+1}$-smooth and thus the surface identity map $\id_S$ is $C^{l+1}$-smooth.
By \eqref{eqn:PisDid} $P = D_S \id_S$ is $C^{l}$-smooth.
Therefore we can handle differential operators of higher order by
iterating Theorem~\ref{theo:cpopDiff} even though $\cp$ is only
$C^1$-smooth because we never differentiate $\cp$ more than once.

In the next section, we will discuss special second order differential operators where we can drop the second extension.
In that case we differentiate \eqref{eqn:PrinGrad} instead of \eqref{eqn:GradReext} and hence we need a second derivative of $\cp$.
There we will require $\cp$ to be a $C^2$-smooth closest point function.

\section{Surface Intrinsic Diffusion Operators}\label{sect:diffusion}
In this section we discuss the treatment of
diffusion operators, that is, second order differential operators of the form
\[
	\diver_S ( A(y) \nabla_S u),
\]
in terms of the closest point calculus.  Of course, by combining the Gradient and Divergence Principles~\ref{cor:PrinGrad} and~\ref{cor:PrinDiv},
we can set up the  diffusion operator as follows
\begin{align*}
	&  w(x) := A \circ \cp(x) \; \nabla [u \circ \cp](x)\;, \quad x \in B(S), \\
	\Rightarrow \quad &  w(y) = A(y) \nabla_S u(y) \;, \quad y \in S\\
	\Rightarrow \quad &  \diver [w \circ \cp](y) = \diver_S w(y) = \diver_S ( A(y) \nabla_S u(y)) \;.
\end{align*}
In this set-up, we have a second extension $w \circ \cp$ in the last step.
If the surface vector field $v = A(y) \nabla_S u$ is tangent to the surface, we call the operator a \emph{surface intrinsic diffusion operator}
(these are relevant regarding the physical modeling of surface processes).  The subject of this section is that we can drop the second extension
in many cases (depending on $A$ and $\cp$) given the tangency of $v$.
The key to this result is the following lemma on the divergence of vector fields which are tangential on the surface while
tangency is allowed to be mildly perturbed off the surface.

\begin{lem}{(Divergence of tangential fields)}\label{lem:div}
	Let $S$ be a smooth surface, and let $\gamma: U \subset \R^k \to \Omega_S$ ($k = \dim S$) be a regular parametrization of a generic surface patch $\Omega_S \subset S$.
	Let $\Omega$ be a corresponding open subset of the embedding space $\R^n$ such that $\Omega \cap S = \Omega_S$.
	Let $\Phi: U \times V \to \Omega$ ($V \subset \R^{n-k}$) be a coordinate system  on $\Omega$ that satisfies
	\[
		\Phi(\theta_1,\theta_2) = \gamma(\theta_1) + N \circ \gamma(\theta_1) \cdot \theta_2 + \Order( |\theta_2|^2 ) \;, \; \text{as} \quad \theta_2 \to 0 \;.
	\]
	Here the columns of the $n \times (n-k)$ matrix $N(y) := (N_1(y) | \ldots | N_{n-k}(y))$, $y \in S$, give some basis of the normal space  $\Norm_y S$.
	Let $v: B(S) \to \R^n$ be a continuously differentiable vector field and let
	\[
		\bar{v}(\theta_1,\theta_2) = v \circ \Phi(\theta_1,\theta_2) = D_{\theta_1} \gamma(\theta_1) \cdot \tilde{v}(\theta_1,\theta_2) + N \circ \gamma(\theta_1) \cdot \eta(\theta_1,\theta_2)
	\]
	be the representation of this vector field on $\Omega$ in $\theta = (\theta_1,\theta_2)$-variables with a decomposition into the two
	components ($\tilde{v} \in \R^k$, $\eta \in \R^{n-k}$) which are tangential and normal to the surface $S$.

	If the coefficient $\eta$ of the normal part satisfies
	\[
		\eta(\theta_1,0) = 0 \quad \text{and} \quad \trace D_{\theta_2} \eta(\theta_1,0) = 0 \;,
	\]
	then (as the surface patch is generic)
	\[
		\diver v(y) = \diver_S v(y) \;, \quad y \in S \;.
	\]
	The condition $\eta(\theta_1,0) = 0$ means that the restriction $v|_S$ is a tangential surface vector field.
\end{lem}

\begin{proof}
	We show $\diver_S v \circ \gamma = \diver v \circ \gamma$. For a tangential field the surface divergence in terms of a parametrization $\gamma$ is given by
	\begin{equation}\label{eqn:divSV}
		\begin{aligned}
			\diver_S v \circ \gamma &= \frac{1}{\sqrt{g}} \diver_{\theta_1} (\tilde{v}(\theta_1,0)\sqrt{g})
			= \diver_{\theta_1} (\tilde{v}(\theta_1,0)) + \sum\limits_{l=1}^{k} \frac{\pd_l g}{2 g} \tilde{v}_l(\theta_1,0) \;.
		\end{aligned}
	\end{equation}
	where $g=\det G$ and $G = D_{\theta_1} \gamma^T D_{\theta_1} \gamma$ is the metric tensor induced by $\gamma$.

	Next, we turn to $\diver v \circ \gamma$. In a first step, we obtain
	\begin{align*}
		\diver v \circ \gamma &= \trace Dv \circ \gamma = \trace Dv \circ \Phi(\theta_1,0) = \trace \left( D_\theta \bar{v}(\theta_1,0) \cdot D_\theta \Phi(\theta_1,0)^{-1} \right) \;.
	\end{align*}
	The derivative of $\bar{v}$ is
	\begin{align*}
		D_\theta \bar{v}(\theta_1,\theta_2)
		&= \sum\limits_{l=1}^k  \pd_l \gamma \nabla_\theta \tilde{v}_l^T
			+ \sum\limits_{l=1}^{n-k}  N_l \circ \gamma \nabla_\theta \eta_l^T
			+ \sum\limits_{l=1}^k  \tilde{v}_l \; \pd_l D_\theta \gamma
			+ \sum\limits_{l=1}^{n-k}  \eta_l \; D_\theta (N_l \circ \gamma)  \;. \\
		D_\theta \bar{v}(\theta_1,0)
		&= \sum\limits_{l=1}^k  \pd_l \gamma \nabla_\theta \tilde{v}_l^T
			+ \sum\limits_{l=1}^{n-k}  N_l \circ \gamma \nabla_\theta \eta_l^T
			+ \sum\limits_{l=1}^k  \tilde{v}_l \; \pd_l D_\theta \gamma
	\end{align*}
	since $\eta(\theta_1,0) = 0$. By our assumptions on $\Phi$ we have
	\[
		D_\theta \Phi(\theta_1,0) = ( D_{\theta_1} \gamma \; | \; N \circ \gamma) \;, \quad D_\theta \Phi(\theta_1,0)^{-1} =
		( D_{\theta_1} \gamma \; | \; N \circ \gamma)^{-1} =
		\left( \begin{matrix}
					D_{\theta_1} \gamma^{\dagger} \\
					\hline
					N \circ \gamma^{\dagger}
					\end{matrix} \right) \;.
	\]
	Now, we write the divergence as
	\begin{align*}
		\diver v \circ \gamma
		&= \sum\limits_{l=1}^k \trace\left( \pd_l \gamma \nabla_\theta \tilde{v}_l^T \; \cdot D_\theta \Phi(\theta_1,0)^{-1}\right)
			+ \sum\limits_{l=1}^{n-k} \trace\left( N_l \circ \gamma \nabla_\theta \eta_l^T \; \cdot D_\theta \Phi(\theta_1,0)^{-1}\right) \\
		&	+ \sum\limits_{l=1}^k \trace\left( \pd_l D_\theta \gamma \; \cdot D_\theta \Phi(\theta_1,0)^{-1}\right) \tilde{v}_l \;.
	\end{align*}
	Using the rules of the $\trace$ operator we get
	\begin{align*}
		\diver v \circ \gamma &
		= \trace \left( D_\theta \left( \begin{matrix} \tilde{v} \\ \eta \end{matrix}\right) D_\theta \Phi(\theta_1,0)^{-1} ( D_{\theta_1} \gamma \; | \; N \circ \gamma)  \right)
		+ \sum\limits_{l=1}^k \trace\left( \pd_l D_\theta \gamma \; \cdot D_\theta \Phi(\theta_1,0)^{-1}\right) \tilde{v}_l \;.
	\end{align*}
	Next, we use $D_\theta \gamma = ( D_{\theta_1} \gamma \; | \; 0 )$ and the second assumption $\trace D_{\theta_2}\eta(\theta_1,0) = 0$ to get
	\begin{equation}\label{eqn:divV}
		\begin{aligned}
			\diver v \circ \gamma &
			= \trace \left( D_\theta \left( \begin{matrix} \tilde{v} \\ \eta \end{matrix}\right) \right)
			+ \sum\limits_{l=1}^k \trace\left( \pd_l D_{\theta_1} \gamma \cdot D_{\theta_1} \gamma^{\dagger} \right) \tilde{v}_l \\
			&= \trace \left( D_{\theta_1} \tilde{v} \right)
			+ \sum\limits_{l=1}^k \trace\left( \pd_l D_{\theta_1} \gamma \cdot D_{\theta_1} \gamma^{\dagger} \right) \tilde{v}_l \\
			&= \diver_{\theta_1} (\tilde{v}(\theta_1,0)) + \sum\limits_{l=1}^k \trace\left( \pd_l D_{\theta_1} \gamma \cdot D_{\theta_1} \gamma^{\dagger} \right) \tilde{v}_l
		\end{aligned}
	\end{equation}
	The final step is to show that $\trace\left( \pd_l D_{\theta_1} \gamma \cdot D_{\theta_1} \gamma^{\dagger} \right) = \frac{\pd_l g}{2 g}$.
	Partial differentiation of the definition of $G = D_{\theta_1} \gamma^T D_{\theta_1} \gamma$ and using the definition
	of $D_{\theta_1} \gamma^{\dagger} = G^{-1} D_{\theta_1} \gamma^T$ yields
	\begin{equation}\label{eqn:pseudoinv}
		\trace\left( \pd_l D_{\theta_1} \gamma \cdot D_{\theta_1} \gamma^{\dagger} \right) = \frac{1}{2} \trace(G^{-1} \pd_l G)
	\end{equation}
	While partial differentiation of the equation $G^{-1}G=I$ results in the (matrix) ordinary differential equation (ODE)
	\begin{equation*}\label{eqn:metricT}
		\pd_l \left( G^{-1} \right) = -(G^{-1} \pd_l G) \; G^{-1} \;,
	\end{equation*}
	which implies $\pd_l \det G^{-1} = - \trace(G^{-1} \pd_l G) \det G^{-1}$ and
	hence $g=\det G$ satisfies
	\[
		-\frac{1}{g^2} \pd_l g = \pd_l \frac{1}{g}= \pd_l \det G^{-1} = -\trace(G^{-1} \pd_l G) \det G^{-1} = - \trace(G^{-1} \pd_l G) \frac{1}{g} \;.
	\]
	This result combined with \eqref{eqn:pseudoinv} completes the proof: we replace the last term in the last equality of \eqref{eqn:divV} with
	\[
		\trace\left( \pd_l D_{\theta_1} \gamma \cdot D_{\theta_1} \gamma^{\dagger} \right) = \frac{1}{2} \trace(G^{-1} \pd_l G) = \frac{\pd_l g}{2 g} \;,
	\]
	and compare with \eqref{eqn:divSV} to see that $\diver_S v \circ \gamma = \diver v \circ \gamma$.
	\hfill
\end{proof}

\subsection{Diffusion Operators using Fewer Extensions}
Armed with Lemma \ref{lem:div} we obtain a closest point calculus involving fewer extensions, for all surface intrinsic diffusion operators.
There are two situations depending on how much ``help'' we get from the diffusion tensor $A$.
First, we look at the situation where $A(y)$ maps all vectors $\xi \in \R^n$ to tangent vectors.

\begin{theo}\label{theo:tensor}
	Let $A:S \to \R^{n \times n}$ be a tangential diffusion tensor field, that maps all vectors to tangent vectors, i.e.,
	$\forall y \in S$ we have $\forall \xi \in \R^n \; \Rightarrow \; A(y) \xi \in \Tang_y S$.
	Then the corresponding surface diffusion operator can be written in terms of a closest point extension as
	\[
		\diver_S \left( A \nabla_S u\right)(y) = \diver \left( A \circ \cp \nabla[u \circ \cp] \right)(y) \;,\quad y \in S \;.
	\]
	A re-extension of $\nabla[u \circ \cp]$ is not necessary and the result is true for all $C^2$-smooth closest point functions.
\end{theo}
\begin{proof}
	As $A$ maps all vectors to tangent vectors the field $v := A \circ \cp \nabla[u \circ \cp]$ is tangential to $S$ for all $x$.
	Now, we want to apply Lemma \ref{lem:div}, that means we have to define the coordinate system $\Phi$ and the field $\bar{v}(\theta_1,\theta_2)$.
	We reuse the parametrization $\gamma$ of a generic surface patch from Lemma \ref{lem:div}.
	For a fixed value of $\theta_1$ and the corresponding surface point $\gamma(\theta_1)$ we simply parametrize (with $\theta_2$) the pre-image $\cp^{-1}( \gamma(\theta_1) )$
	in order to get $\Phi(\theta_1,\theta_2)$. Since by Theorem \ref{theo:CPbyGeo} the pre-image $\cp^{-1}( \gamma(\theta_1) )$ intersects $S$ orthogonally, we can organize
	its parametrization in such a way that $\Phi$ satisfies
	\[
		\Phi(\theta_1,\theta_2) = \gamma(\theta_1) + N \circ \gamma(\theta_1) \cdot \theta_2 + \Order( |\theta_2|^2 ) \;, \; \text{as} \quad \theta_2 \to 0 \;.
	\]
	As our particular vector field $v$ is tangential to $S$ for all $x$ in a band $B(S)$ around $S$, the corresponding $\bar{v}$ from Lemma \ref{lem:div} takes the form
	\[
		\bar{v}(\theta_1,\theta_2) = v \circ \Phi(\theta_1,\theta_2) = D_{\theta_1} \gamma(\theta_1) \cdot \tilde{v}(\theta_1,\theta_2)
	\]
	with $\eta$ equal to zero. Hence, by Lemma \ref{lem:div} we conclude the first equality of
	\[
		\diver( A \circ \cp \nabla [u \circ \cp] )(y) = \diver_S( A \circ \cp \nabla [u \circ \cp] )(y) = \diver_S( A \nabla_S u )(y) \;, \quad y \in S \;,
	\]
	while the second equality is because
	we have $\cp|_S = \id_S$ and $ \left.\nabla [u \circ \cp]\right|_S = \nabla_S u$.
	\hfill
\end{proof}
\medskip

Next, we look at the situation where $A(y)$ maps only tangent vectors $\xi \in \Tang_y S$ to tangent vectors.
This case is particularly interesting since---by considering diffusion tensors of the form
\begin{equation}\label{eqn:ScalarDiffCoeff}
	A(y) = a(y) \cdot I \;,
\end{equation}
where $I \in \R^{n \times n}$ is the identity matrix on the embedding space---it covers the special case where we are given a scalar diffusion coefficient $a(y) \in \R$.
It is clear that the diffusion tensor in \eqref{eqn:ScalarDiffCoeff} can only map tangent vectors to tangent vectors
(this is weaker than the requirement in Theorem~\ref{theo:tensor}).

\begin{theo}\label{theo:tensor2}
	Let $A:S \to \R^{n \times n}$ be a tangential diffusion tensor field, that maps only tangent vectors to tangent vectors, i.e.,
	$\forall y \in S$ we have $\forall \xi \in \Tang_y S \; \Rightarrow \; A(y) \xi \in \Tang_y S$.
	If $\cp$ is a $C^2$-smooth closest point function such that the transpose of its Jacobian maps tangent vectors to tangent vectors, i.e.,
	\[
		\forall x \in B(S) \quad \text{we have:} \quad \forall \xi \in \Tang_{\cp(x)} S \; \Rightarrow \; D \cp(x)^T \xi \in \Tang_{\cp(x)} S \;.
	\]
	then the corresponding surface diffusion operator can be written in terms of a closest point extension as
	\[
		\diver_S \left( A \nabla_S u\right)(y) = \diver ( A \circ \cp \nabla[u \circ \cp] )(y) \;,\quad y \in S \;.
	\]
	A re-extension of $\nabla[u \circ \cp]$ is not necessary.
\end{theo}

\begin{proof}
	As $A$ maps only tangent vectors to tangent vectors the field $v := A \circ \cp \nabla[u \circ \cp]$ will be tangential to $S$
	only if $\nabla[u \circ \cp]$ is tangential to $S$. Since $\cp$ is a retraction its Jacobian satisfies (compare to \eqref{eqn:PDr})
	\begin{equation}\label{eqn:tang}
		D\cp = P \circ \cp D\cp \;.
	\end{equation}
	Now, we expand $\nabla[u \circ \cp]$ by referring to an arbitrary extension $u_E$ of $u$ as of Definition \ref{def:ext}:
	\[
		 \nabla [u \circ \cp]  =  D \cp^T \nabla u_E \circ \cp =  D \cp^T P \circ \cp \nabla u_E \circ \cp = D \cp^T \nabla_S u \circ \cp\;.
	\]
	Since $\xi := \nabla_S u \circ \cp(x)$ belongs to $\Tang_{\cp(x)} S$, $\nabla[u \circ \cp](x)$ also belongs to $\Tang_{\cp(x)} S$ by our requirement on $D \cp(x)^T$.
	Now, we are sure that $v := A \circ \cp \nabla[u \circ \cp]$ is tangential to $S$ for all $x$, and so the same arguments
	used in the proof of Theorem~\ref{theo:tensor} apply. \hfill
\end{proof}

\begin{prop}
	If $\cp$ is a $C^2$-smooth closest point function with a symmetric
	Jacobian $D \cp = D \cp^T$ then $\cp$ satisfies the requirements of Theorem~\ref{theo:tensor2}.
	The Euclidean closest point function $\ecp$ is one such closest point function with
	symmetric Jacobian
	\[
		D \ecp(x) =
		\begin{cases}
			I - \nabla d(x) \cdot \nabla d(x)^T - d(x) D^2 d(x) & x \in B(S) \setminus S, \\
			P(x) & x \in S.
		\end{cases}
	\]
\end{prop}
\begin{proof}
	Equation \eqref{eqn:tang} combined with symmetry of the Jacobian yields
	\[
		P \circ \cp D\cp = D\cp = D \cp^T
	\]
	which shows that $D \cp^T$ in this case maps all vectors $\xi$ to tangent vectors.
	For the derivation of $D\ecp$ see the proof of Theorem~\ref{theo:ECPisCP}.
	For further smoothness properties of $\ecp$ see Section~\ref{subsubsect:RemEcp}.
	\hfill
\end{proof}

\subsubsection{More General Diffusion Coefficients}
In both Theorems~\ref{theo:tensor} and~\ref{theo:tensor2} the crucial bit is that $A(y)$ for fixed $y$ maps certain vectors $\xi$ to the tangent space, i.e., $A(y) \xi \in \Tang_y S$.
The theorems are still true if we let $A$ also depend on the function $u$, for example this dependence could be of the form $A(y,u(y),\nabla_S u(y))$.

\subsection{The Laplace--Beltrami operator} As an example we discuss the surface Laplacian (the Laplace--Beltrami operator) in terms of the closest point calculus:
\begin{enumerate}[1.]
	\item We may write the Laplace--Beltrami operator as
			\[
				\laplace_S u = \diver_S ( \nabla_S u) \;.
			\]
			Given a closest point function that satisfies the requirement of Theorem \ref{theo:tensor2}, e.g., the Euclidean closest point function $\ecp$, we have
			\[
				\laplace [u \circ \cp](y) = \laplace_S u(y) \; , \quad y \in S \;.
			\]
	\item Alternatively, we may rewrite the Laplace--Beltrami operator by using a diffusion tensor as
			\[
				\laplace_S u = \diver_S ( P \nabla_S u) \;.
			\]
			The diffusion tensor here is the projector onto the tangent space. So we have replaced the identity matrix with the surface intrinsic identity matrix.
			Now Theorem~\ref{theo:tensor} allows us to compute the Laplace--Beltrami operator like
			\[
				\diver ( P \circ \cp \nabla[u \circ \cp] )(y) = \laplace_S u(y) \; , \quad y \in S \;,
			\]
			without any further requirements on $D \cp$.
	\item If $P$ is not known a priori and $\cp$ does not satisfy the requirement of Theorem \ref{theo:tensor2}, we can always work with re-extensions
			(directly applying the Gradient and Divergence Principles~\ref{cor:PrinGrad} and~\ref{cor:PrinDiv})
			\[
				\diver( \nabla [u \circ \cp] \circ \cp)(y) = \laplace_S u(y) \; , \quad y \in S.
			\]
			Note by expanding the expression $\nabla [u \circ \cp] \circ \cp$ as in \eqref{eqn:GradReext}
			\begin{align*}
				\nabla [u \circ \cp] \circ \cp &= P\circ \cp \nabla u_E \circ \cp,
			\end{align*}
			we can see that re-extension in fact means an implicit version of approach 2.
\end{enumerate}

\subsection{Proof of a Principle of Ruuth \& Merriman}
For the original Closest Point Method (with $\ecp$) Ruuth \& Merriman in \cite{sjr:SimpleEmbed} also reasoned that a re-extension is not necessary
for some diffusion operators
based on the fact that the Euclidean closest point extension satisfies the PDE
\[
	\left< \, \nabla d(x) , \nabla [u \circ \ecp] \, \right> = 0 \; , \quad x \in B(S) \setminus S
\]
and the special principle:
\begin{quote}
  ``Let $v$ be any vector field on $\R^n$ that is tangent at $S$ and
  also tangent at all surfaces displaced by a fixed distance from $S$
  (i.e., all surfaces defined as level sets of the distance function
  $d$ to $S$).  Then at points $y$ on the surface $\diver_S v(y) =
  \diver v(y)$.''
\end{quote}
We give a proof of this below as a consequence of Lemma~\ref{lem:div}
but with the additional assumption of $C^2$-regularity of the vector field $v$ if $\codim S \geq 2$.

Ruuth \& Merriman \cite{sjr:SimpleEmbed} also use this principle to establish a divergence principle.
This requires that the surface vector field $w$ be tangent to the surface and it also requires the use of the Euclidean closest point function to extend the surface vector field to the embedding space as $v = w \circ \ecp$.
In contrast, the more general Divergence Principle~\ref{cor:PrinDiv} works for any (possibly non-tangential) surface vector field with extensions using any closest point function.

As said, their principle can also be used with the Euclidean closest point
function to allow fewer extensions in certain surface diffusion operators.
Theorem~\ref{theo:tensor2} generalizes this, allowing for a larger class of
closest point functions.

\begin{proof}
	Given the parametrization $\gamma$ of a generic surface patch, the map
	\[
		\Phi(\theta_1,\theta_2) = \gamma(\theta_1) + N \circ \gamma(\theta_1) \cdot \theta_2
	\]
	parametrizes $\ecp^{-1}(\gamma(\theta_1))$ for fixed $\theta_1$ and gives us the required coordinate system.
	For points $x$ off the surface the normal $\nabla d$ to the level-sets of the distance map is given by
	\[
		\nabla d(x) = \frac{x-\ecp(x)}{|x-\ecp(x)|} \qquad \Rightarrow \qquad \nabla d \circ \Phi(\theta_1,\theta_2) = \frac{N \circ \gamma(\theta_1) \cdot \theta_2}{|N \circ \gamma(\theta_1) \cdot \theta_2|}.
	\]
	Then the image of the following projection matrix
	\[
		Q(x) = I-\nabla d(x) \nabla d(x)^T \qquad \Rightarrow \qquad \bar{Q}(\theta_1,\theta_2) = Q \circ \Phi(\theta_1,\theta_2) = I - \frac{N \theta_2 \theta_2^T N^T}{|N \cdot \theta_2|^2}
	\]
	is tangent to level lines of the Euclidean distance map $d$. Since $v$ is itself tangent to the latter level lines, we have
	$Q \cdot v = v$ or in $\theta$-variables $\bar{Q} \bar{v} = \bar{v}$ with
	\[
		\bar{v}(\theta_1,\theta_2) = v \circ \Phi(\theta_1,\theta_2) = D_{\theta_1} \gamma(\theta_1) \cdot \tilde{v}(\theta_1,\theta_2) + N \circ \gamma(\theta_1) \cdot \eta(\theta_1,\theta_2)
	\]
	as in Lemma \ref{lem:div}. Since $\nabla d$ is also normal to the surface $S$, the product $\bar{Q} \bar{v}$ is
	\[
		\bar{Q} \bar{v} = D_{\theta_1} \gamma \cdot \tilde{v} + \bar{Q} N \eta \;.
	\]
	If the codimension is one, the matrix $N$ is an $n \times 1$ matrix and $\nabla d = \pm N$, hence, the summand $\bar{Q} N \eta$
   cancels out, and this product reduces to $\bar{Q} \bar{v} =  D_{\theta_1} \gamma \cdot \tilde{v}$.
	Consequently, in order for $\bar{Q} \bar{v} = \bar{v}$ to hold, $\eta \equiv 0$ must vanish, and Lemma \ref{lem:div} yields the result.
	If the codimension is higher, we rewrite the product as
	\[
		\bar{Q} \bar{v} = D_{\theta_1} \gamma \cdot \tilde{v} + N \cdot \left( I - \frac{\theta_2 \theta_2^T N^TN }{|N \cdot \theta_2|^2} \right) \eta \;.
	\]
	The coefficient must satisfy $\eta(\theta_1,0)=0$ so that $v(x)$ is tangent to the surface if $x \in S$.
	In order to have additionally $\bar{Q} \bar{v} = \bar{v}$, $\eta$ must also satisfy
	\[
		\theta_2^T A \eta = 0 \quad \text{where} \quad A = A(\theta_1) := (N\circ \gamma(\theta_1))^T \cdot N\circ \gamma(\theta_1) \;.
	\]
	We differentiate the new condition $\theta_2^T A \eta = 0$ with respect to $\theta_2$ and obtain
	\[
		\eta^T A  + \theta_2^T A D_{\theta_2} \eta = 0 \qquad \Rightarrow \qquad  \eta = - A^{-1} D_{\theta_2} \eta^T A \theta_2
	\]
	Since in higher codimension $v$ is assumed to be $C^2$, we can differentiate again which yields
	\[
		D_{\theta_2} \eta = - A^{-1} D_{\theta_2} \eta^T A  - A^{-1} D^2_{\theta_2} \eta^T A \theta_2 \;.
	\]
	We apply the trace-operator and get
	\[
		\trace D_{\theta_2} \eta = - \trace D_{\theta_2} \eta^T - \trace \left( A^{-1} D^2_{\theta_2} \eta^T A \theta_2 \right) \;.
	\]
	Finally, since the second derivative $D^2_{\theta_2} \eta$ is bounded on a small compact neighborhood of $(\theta_1,0)$ we have
	\[
		\trace D_{\theta_2} \eta = \Order(|\theta_2|) \;, \quad \theta_2 \to 0 \;.
	\]
	So the second requirement on $\eta$ is satisfied and Lemma \ref{lem:div} yields the result. \hfill
\end{proof}

\section{Construction of Closest Point Functions from Level-Set Descriptions}\label{sect:Const}
Beginning with the case of $\codim S =1$, we present a general
construction of closest point functions in the special case when the
surface is given by a level set (or as an intersection of several level sets.)

\subsection{Codimension One}\label{subset:codim1}
Let $\varphi: B(S) \to S$ denote a $C^2$-smooth scalar level-set function.
The surface $S$ is the zero-level of $\varphi$,
which we assume is a proper implicit description of $S$, i.e.,
\begin{equation}\label{eqn:LSetProper}
	\nabla \varphi(x) \neq 0 \quad \text{for all} \quad x \in B(S) \;.
\end{equation}
Thus the normals to level-sets are given by $N: B(S) \to \R^n$, $N = \nabla \varphi/|\nabla \varphi|$
and the normal field of $S$ is $N|_{S} : S \to \R^n$.
Again $B(S) \subset \R^n$ denotes a band around $S$ and condition \eqref{eqn:LSetProper} will determine a reasonable band $B(S)$
when starting out with a $\varphi$ defined on all of $\R^n$.

Now closely following \cite{TAM:Diss}, we construct a retraction by solving the following initial value problem (IVP):
\begin{equation}\label{eqn:bmoc}
	\xi' = - \nabla \varphi \circ \xi \;, \qquad \xi(0,x) = x \;.
\end{equation}
We denote the corresponding family of trajectories by $\xi(\tau, x)$, where $\tau$ is the ODE-time, while the parameter $x$ refers to the initial point.
If we start at a point $x$ with $\varphi(x) > 0$ and solve forward in ODE-time,
$\xi$ is a steepest descent trajectory, while if we start at a point $x$ with $\varphi(x) < 0$
and solve backward in ODE-time, we will obtain a steepest ascent trajectory heading for the surface.
The unique intersection of the trajectory $\xi(\dotarg,x)$ with the surface $S$ defines a retraction that maps the initial point $x$
to some point on $S$.  The next step is to find the point of intersection which we achieve by a suitable transformation of ODE \eqref{eqn:bmoc}.
We consider the descent case ($\varphi(x) > 0$), and relate the ODE-time $\tau$ and the level label $\lambda$ by
\begin{equation}\label{eqn:between}
	\varphi( \xi(\tau(\lambda),x) ) = \varphi(x)-\lambda \;.
\end{equation}
The implicit function $\tau:[0,\varphi(x)] \to \R_+$, $\lambda \to \tau(\lambda)$ takes the value $\tau(0)=0$
and has the derivative $\tau' = 1/|\nabla \varphi|^2 \circ \xi$.
The transformation that we want is $\eta(\lambda,x) := \xi(\tau(\lambda),x)$
because \eqref{eqn:between} can be rewritten as $\varphi( \eta(\lambda,x) ) = \varphi(x)-\lambda$, and
thus evaluating $\eta$ at $\lambda = \varphi(x)$ returns the corresponding point on the surface $S$:
\[
	\varphi( \eta(\varphi(x),x) ) = 0 \quad \Leftrightarrow \quad \eta(\varphi(x),x) \in S.
\]
We obtain $\eta$ as the solution of the new initial value problem (see also \cite{TAM:Diss})
\begin{equation}\label{eqn:bmoc2}
	\eta'(\lambda,x) = -\frac{\nabla \varphi}{|\nabla \varphi|^2} \circ \eta(\lambda,x) \;, \qquad \eta(0,x) = x \;.
\end{equation}
So far we have considered the descent case but for ascent we end up with the same IVP.
Finally, we obtain the desired closest point function $\cp$ by
\[
	\cp: B(S) \to S \;, \qquad \cp(x) := \eta(\varphi(x),x) \;.
\]
By construction, this function is a $C^1$-retraction (because the right-hand side of ODE \eqref{eqn:bmoc2} is $C^1$-smooth
the differentiability of $\eta(\tau,x)$ with respect to $x$ follows from ODE-theory, see e.g., \cite[Chapter 4.6]{KK:Ana2}).
And that it is in fact a closest point function is a direct consequence of Theorem~\ref{theo:CPbyGeo}:
because of our construction the pre-image $\cp^{-1}(y)$ is exactly a trajectory of the ODE \eqref{eqn:bmoc2}
and this trajectory intersects $S$ orthogonally.

\subsubsection{Remarks}\label{subsubsect:Remarks}
\begin{enumerate}[$\bullet$]
   \item If one is only interested in the construction of a retraction map it suffices
			to replace $\nabla \varphi$ in \eqref{eqn:bmoc2} with a transversal field $c$, i.e.,
			$|c| = 1$ and $\left< c, N \right> \geq \beta > 0$. After a similar transformation
			the retraction is given by $r(x) = \eta(\varphi(x),x)$.
			 This approach to retractions is actually the method of backward characteristics (see e.g., \cite{TAM:Diss}):
			the solution of the PDE
			\begin{equation}\label{eqn:extensionPDE}
				\left< c(x), \nabla v \right> = 0 \;, \quad v|_S = u \;,
			\end{equation}
			is $v = u \circ r$. So, in a neighborhood of $S$, $v$ defines an extension of the surface function $u$
			and is as smooth as $u$ (see the method of characteristics in e.g., \cite{LCE:PDE}).
		   The idea of extending surface functions by solving \eqref{eqn:extensionPDE} with $c=\nabla \varphi$
			has been used in other earlier works, see e.g., \cite{BCO:VPaPDEoIS}.
	\item Note that the same ODE as in \eqref{eqn:bmoc2} is used in \cite{MWH:DiffTopo} for the construction of diffeomorphisms.
	\item If $f: \R \to \R$ is a smooth function with $f' \neq 0$ then the closest point function obtained from using $\hat{\varphi} := f \circ \varphi$
			in IVP \eqref{eqn:bmoc2} is the same as that obtained from the original $\varphi$ because the corresponding ODEs parametrize the same curve through the initial point $x$.
	\item The Euclidean closest point function is a special case our construction:
			let $\varphi = \sd$ be the signed Euclidean distance function.
			In this case, the IVP for $\eta$ reduces to
			$\eta'(\lambda,x) = - \nabla \sd(x)$, $\eta(0,x) = x$,
			with a right-hand side independent of $\lambda$.
			This has the solution $\eta(\lambda,x) = x - \lambda \nabla \sd(x)$
			and the corresponding closest point function is the Euclidean one
			\[
				\eta(\sd(x),x) = x - \sd(x) \nabla \sd(x) = \ecp(x)\;.
			\]
\end{enumerate}

\subsection{Higher Codimension}
In the case of higher codimension, $\codim S = m \geq 2$, we assume that $S$ is the proper intersection of $m$ surfaces of codimension one.
Given this, we construct the closest point function $\cp = \cp_m \circ \cdots \circ \cp_2 \circ \cp_1$ as the composition of $m$ closest point functions, where $\cp_1$ retracts
onto $S_1$, $\cp_2$ retracts $S_1$-intrinsically onto $S_1 \cap S_2$, $\cp_3$ retracts $S_1 \cap S_2$-intrinsically onto $S_1 \cap S_2 \cap S_3$, and so on.
We demonstrate only the situation where $\codim S = 2$, since this is essentially the inductive step for the general case.

Let again $B(S) \subset \R^n$ be a band around $S$ and let $\varphi_j: B(S) \to S$, $j \in \{1,2\}$, be  $C^2$-smooth level-set functions.
Each level-set function $\varphi_j$ shall yield a proper implicit description of a codimension-one surface $S_j$ as its zero-level, thus we require:
\[
	 \nabla \varphi_j(x) \neq 0 \;, \quad \text{for all} \quad x \in B(S) \;.
\]
The two surfaces $S_1$ and $S_2$ are orientable with normal vector fields given by $N_j: S_j \to \R^n$, $N_j = \nabla \varphi_j/|\nabla \varphi_j|$.
We assume their intersection $S = S_1 \cap S_2$ to be proper, that is, the normal vectors have to be linearly independent on $S$.
In fact we assume the linear independence to hold on all of $B(S)$ (possibly we have to narrow $B(S)$).
This implies that any two level-sets $S_1^{\mu_1} := \{ x : \varphi_1(x) = \mu_1\}$, $S_2^{\mu_2} := \{ x : \varphi_2(x) = \mu_2\}$
intersect properly (as long as the intersection is non-empty).

Now, we set up a closest point function $\cp = \cp_2 \circ \cp_1$ in two steps.
The first step is the same as in the previous section: $\cp_1 : B(S) \to B(S) \cap S_1$ maps onto the subset
$B(S) \cap S_1$ of the first surface $S_1$ by $\cp_1 = \eta_1(\varphi_1(x),x)$ where $\eta_1$ solves \eqref{eqn:bmoc2} with
$\varphi_1$ in place of $\varphi$.

The second step is more interesting: we construct an $S_1$-intrinsic retraction $\cp_2 : B(S) \cap S_1 \to S$.
The idea is essentially the same as that of the previous set-up, but now it is $S_1$-intrinsic:
we consider the IVP
\begin{equation}\label{eqn:bmoc2_inS1}
	\eta_2'(\lambda,x) = -\frac{\nabla_{S_1} \varphi_2}{|\nabla_{S_1} \varphi_2|^2} \circ \eta_2(\lambda,x) \;, \qquad \eta_2(0,x) = x \;,
\end{equation}
by transforming analogously to the codimension-one case (c.f., \eqref{eqn:bmoc2}).
Note that $\nabla_{S_1} \varphi_2$ is given by $\nabla_{S_1} \varphi_2 = P_1 \nabla \varphi_2$ with $P_1 = I -N_1 N_1^T$
and that the assumption of a proper intersection guarantees that $\nabla_{S_1} \varphi_2$ does not vanish since $N_1$ and $\nabla \varphi_2$ are linearly independent.
Because we start in $S_1$ and $\nabla_{S_1} \varphi_2$ is tangential to $S_1$, the curve $\eta_2(\dotarg,x)$ is contained in $S_1$.
Hence, we obtain an $S_1$-intrinsic closest point function $\cp_2$ by
\[
	\cp_2: B(S) \cap S_1 \to S \;, \qquad \cp_2(x) := \eta_2(\varphi_2(x),x) \;.
\]
The composition $\cp= \cp_2 \circ \cp_1$ defines certainly a $C^1$-retraction $\cp: B(S) \to S$,
and, as all the intersections are orthogonal, Theorem~\ref{theo:CPbyGeo} guarantees $\cp$ to be a closest point function.

\begin{figure}
	\begin{center}
		\includegraphics[width=.43\textwidth]{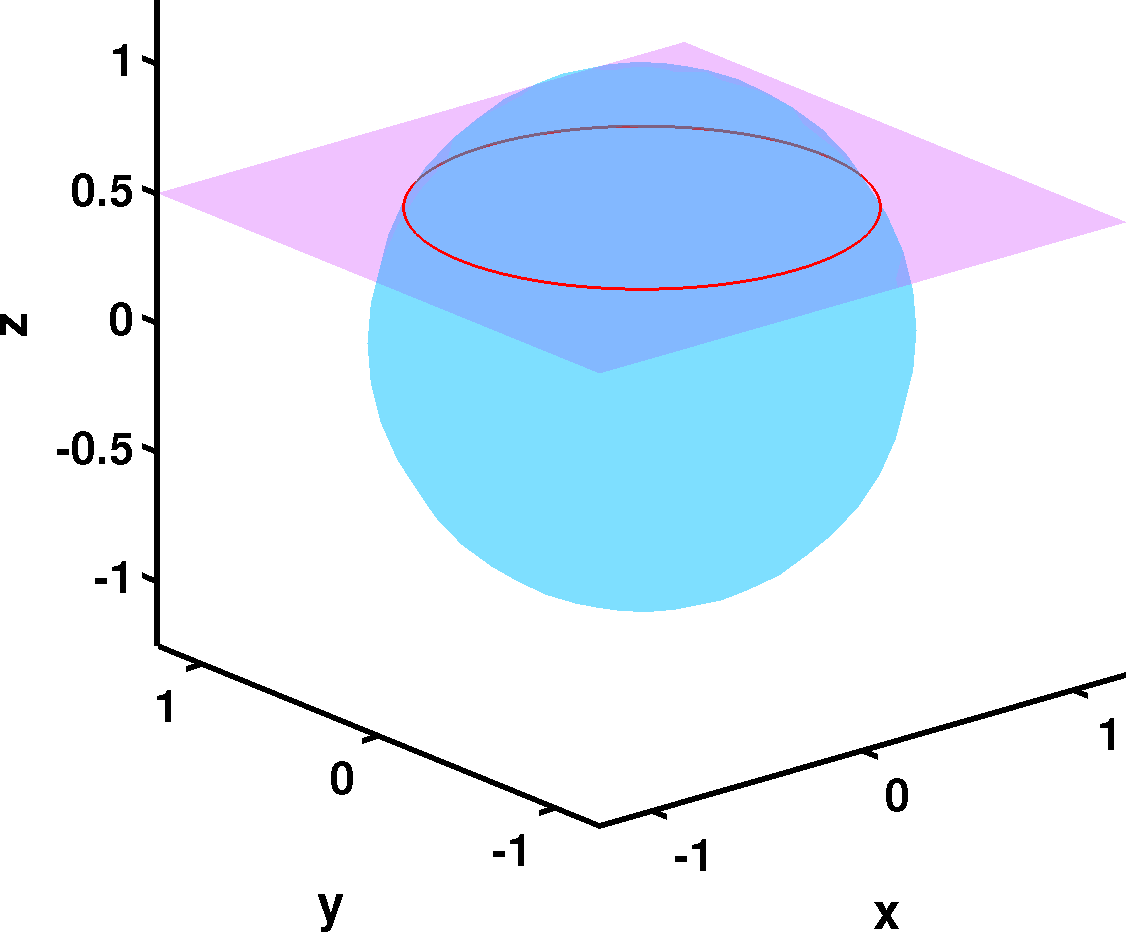}%
                \hspace*{.06\textwidth}%
		\includegraphics[width=.43\textwidth]{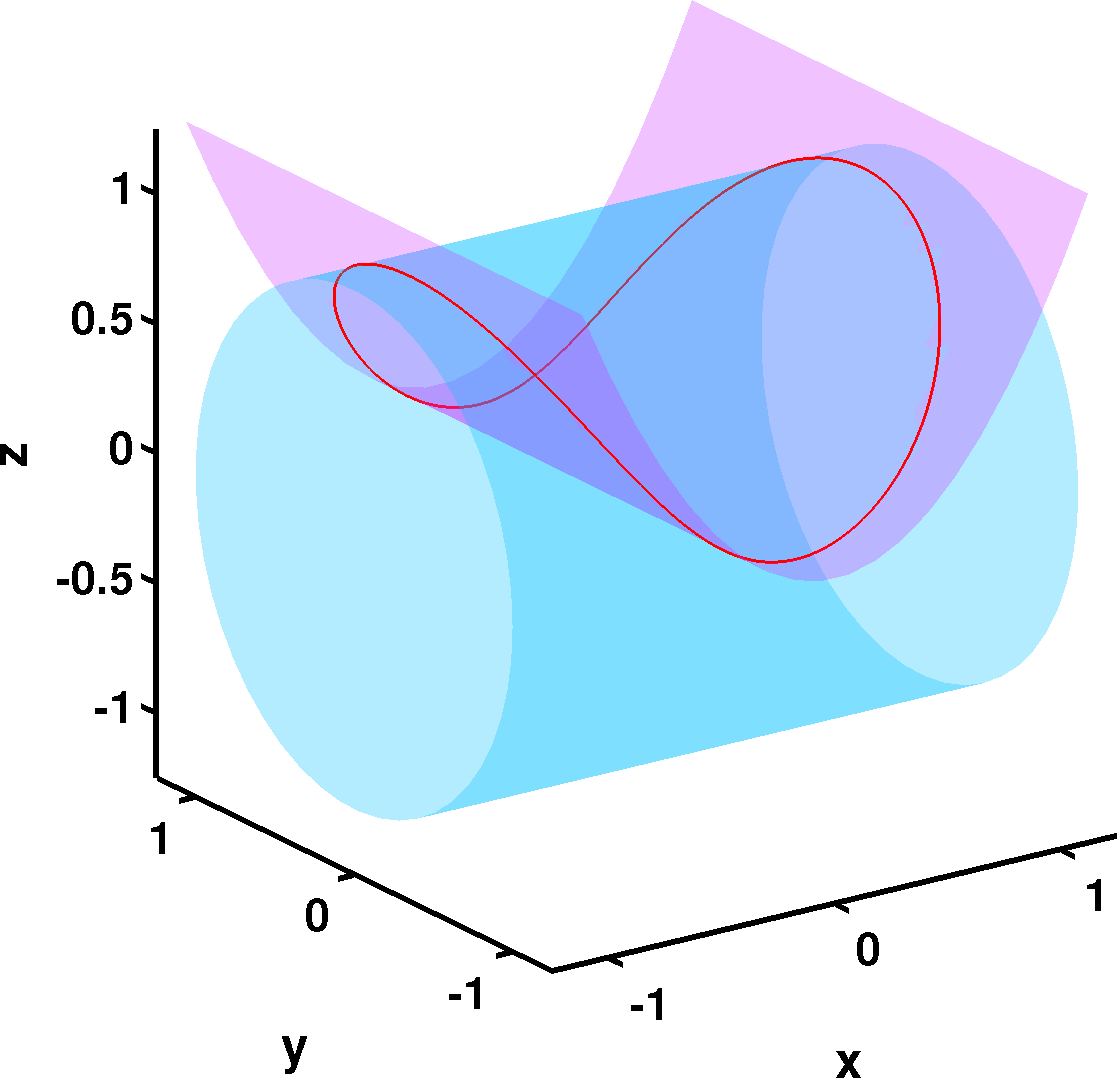}
	\end{center}
	\caption{Example 1 (left): a circle (red) embedded in $\R^3$, given as the intersection of a sphere and a plane.
				Example 2 (right): a curve (red) embedded in $\R^3$, given as the intersection of a cylinder and a parabola.
			}\label{fig:circle_pringle}
\end{figure}

\subsection{Example 1}\label{sect:example1}
We construct a closest point function for a circle $S$ of radius $\sqrt{3}/2$ embedded in $\R^3$ as the intersection of a sphere and a plane (Figure~\ref{fig:circle_pringle} left).
The two level set functions are
\begin{align*}
	\varphi_1 (x_1,x_2,x_3) &= x_1^2 + x_2^2 + x_3^2 -1 \;, &
	\varphi_2 (x_1,x_2,x_3) &= x_3- \tfrac{1}{2} \;.
\end{align*}
The equation $\varphi_1 = 0$ yields the unit sphere as the codimension-one surface $S_1$,
equation $\varphi_2 = 0$ specifies the second codimension-one surface $S_2$ which is a plane parallel to the $x_1 x_2$-plane at $x_3=1/2$,
and the circle $S=S_1 \cap S_2$ is the intersection.

First step: we set up a closest point map $\cp_1$ onto the sphere $S_1$.
Let $c_1 = \nabla \varphi_1$, the IVP for $\eta_1$ is
\begin{equation}\label{eqn:IVP_sphere}
	\eta_1' = -\frac{\nabla \varphi_1}{\left| \nabla \varphi_1 \right|^2} \circ \eta_1 = -\frac{\eta_1}{2 \cdot |\eta_1|^2} \;, \quad \eta_1(0,x) = x \in \R^3 \wo \{0\} \;.
\end{equation}
The maximal band around the sphere where $\nabla \varphi_1 \neq 0$ does not vanish is $\R^3 \wo \{0\}$ .
The solution $\eta_1$ of IVP \eqref{eqn:IVP_sphere} and the corresponding closest point function $\cp_1: \R^3 \wo \{0\} \to S_1$ are
\[
	\eta_1(\lambda,x) = \sqrt{|x|^2 - \lambda} \cdot \frac{x}{|x|} \quad \Rightarrow \quad \cp_1(x) = \eta_1(\varphi_1(x),x) = \frac{x}{|x|} \;.
\]
(Note that $\varphi_1$ is of the form $\varphi_1 = f \circ \sd$ with $f$ strictly monotone. Hence $\cp_1$ is the Euclidean closest point function,
see the remarks \ref{subsubsect:Remarks}.)

Second step: we set up an $S_1$-intrinsic closest point map $\cp_2$ onto the circle $S$.
Let $c_2 = \nabla_{S_1} \varphi_2 = P_1 \cdot \nabla \varphi_2$.
The projector $P_1$ is given by
\[
	P_1(x) = I - \frac{\nabla \varphi_1(x) \cdot \nabla \varphi_1(x)^T}{\left| \nabla \varphi_1(x) \right|^2}=
	\frac{1}{x_1^2 + x_2^2 + x_3^2} \cdot \left(
	\begin{matrix}
		x_2^2 + x_3^2 & -x_1 x_2 & -x_1 x_3 \\
		-x_1 x_2 & x_1^2 + x_3^2 & -x_2 x_3 \\
		-x_1 x_3 & -x_2 x_3 & x_1^2 + x_2^2
	\end{matrix}
	\right) \;.
\]
For $x \in S_1$, i.e., $|x|=1$, we get then
\begin{align*}
	\nabla_{S_1} \varphi_2(x) = P_1(x) \cdot \nabla \varphi_2(x) &=
	\left(
		 -x_1 x_3,\;
		 -x_2 x_3,\;
		 x_1^2 + x_2^2
	\right)^T=
	\left(
		 -x_1 x_3,\;
		 -x_2 x_3,\;
		 1-x_3^2
	\right)^T, \\
	|\nabla_{S_1} \varphi_2(x)|^2 &= 1-x_3^2\;.
\end{align*}
The largest subset of $S_1$ where $\nabla_{S_1} \varphi_2$ does not vanish is the sphere without the poles $S_1 \wo \{\pm e_3\}$, and so the IVP for $\eta_2$ is
\begin{equation}\label{eqn:IVP_sphere_intrinsic}
	\eta_2' = -\frac{\nabla_{S_1} \varphi_2}{\left| \nabla_{S_1} \varphi_2 \right|^2} \circ \eta_2 =
	\left(
	\begin{matrix}
		 \frac{\eta_{2,3}}{1-\eta_{2,3}^2} \cdot \eta_{2,1}, \,
		 \frac{\eta_{2,3}}{1-\eta_{2,3}^2} \cdot \eta_{2,2}, \,
		 -1
	\end{matrix}
	\right)^T \;, \quad \eta_2(0,x) = x \in S_1 \wo \{\pm e_3\} \;,
\end{equation}
where $\eta_{2,j}$ is the $j$-th component of $\eta_2$.
The solution $\eta_2$ of IVP \eqref{eqn:IVP_sphere_intrinsic} and the corresponding closest point function $\cp_2: S_1 \wo \{\pm e_3\} \to S$ are
\begin{align*}
	\eta_2(\lambda,x) &=
	\left(
	\begin{matrix}
		 \sqrt{ \frac{1-(x_3-\lambda)^2}{1-x_3^2} } \cdot x_1, &
		 \sqrt{ \frac{1-(x_3-\lambda)^2}{1-x_3^2} } \cdot x_2, &
		 x_3-\lambda
	\end{matrix}
	\right)^T \\
	\Rightarrow \qquad \cp_2(x) &= \eta_2(\varphi_2(x),x) = \frac{1}{2} \cdot
	\left(
	\begin{matrix}
		 \frac{\sqrt{3}}{ \sqrt{1-x_3^2} } \cdot x_1, &
		 \frac{\sqrt{3}}{ \sqrt{1-x_3^2} } \cdot x_2, &
		 1
	\end{matrix}
	\right)^T\;.
\end{align*}
Finally, we compose the closest point function $\cp: B(S) = \R^3 \wo \{x : x_1 = x_2 = 0\} \to S$ by
\[
	\cp(x) = \cp_2 \circ \cp_1(x) = \frac{1}{2} \cdot
	\left(
	\begin{matrix}
		 \frac{\sqrt{3} x_1}{ \sqrt{x_1^2 + x_2^2} }, &
		 \frac{\sqrt{3} x_2}{ \sqrt{x_1^2 + x_2^2} }, &
		 1
	\end{matrix}
	\right)^T\;.
\]

\begin{figure}
	\begin{center}
		\includegraphics[width=.43\textwidth]{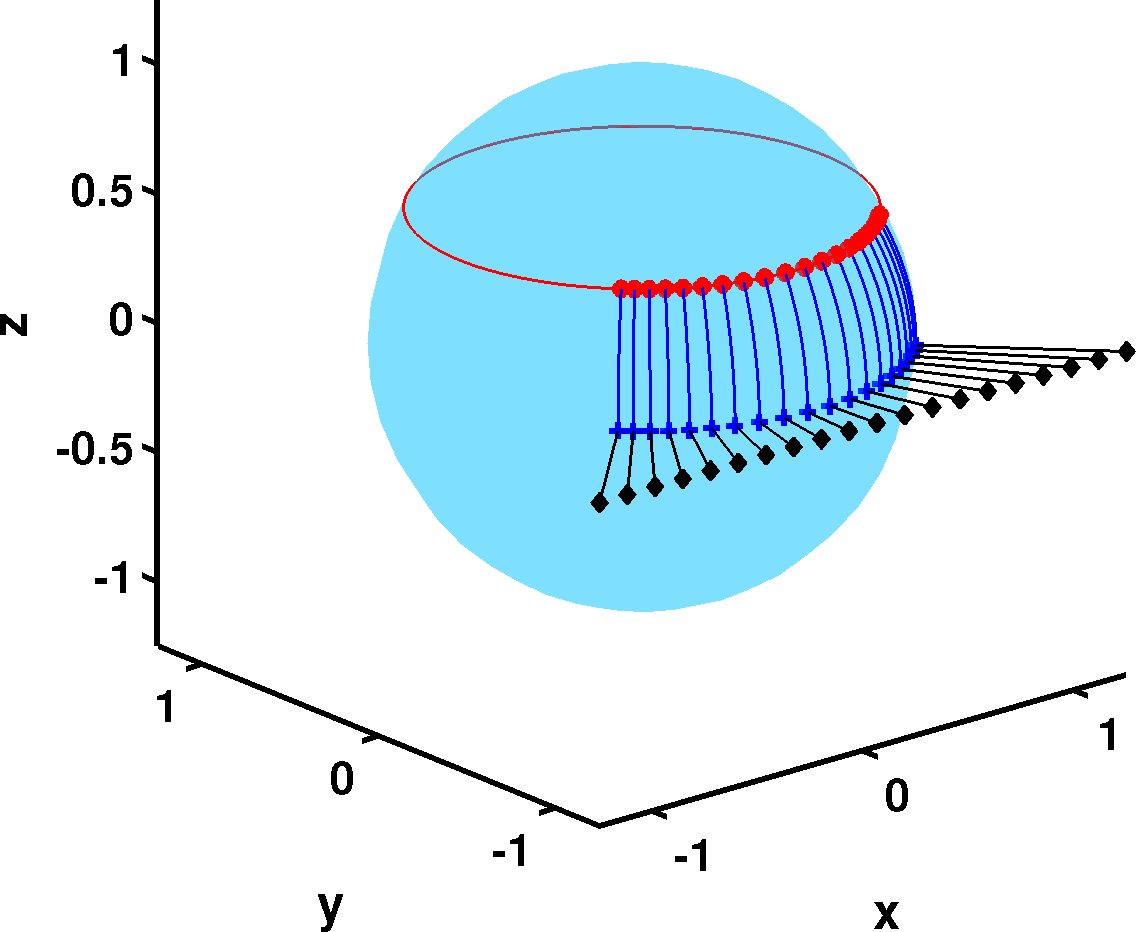}%
                \hspace*{.06\textwidth}%
		\includegraphics[width=.43\textwidth]{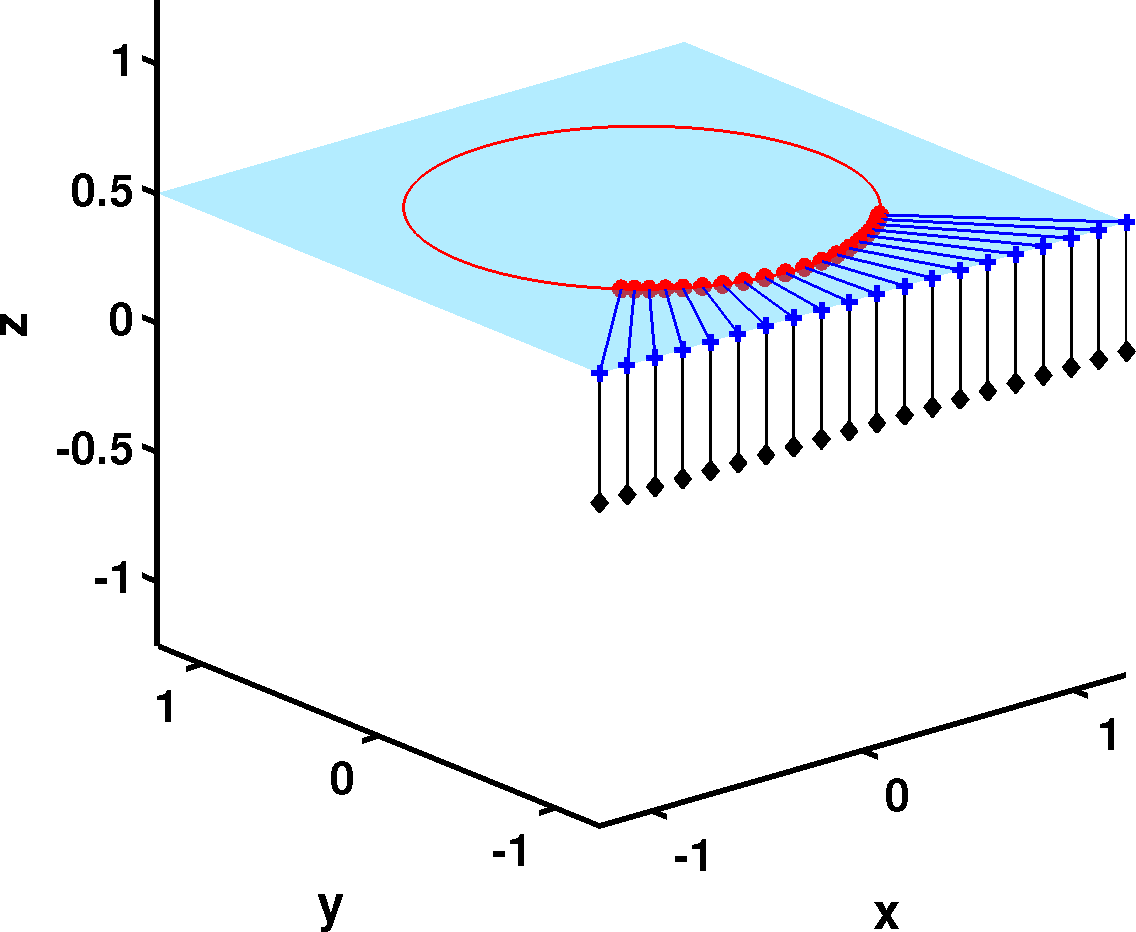}
	\end{center}
	\caption{A circle (red) embedded in $\R^3$, given as the intersection of a sphere and a plane.
				Left: the closest point function $\cp = \cp_2 \circ \cp_1$ maps the black diamonds to the red dots by a two-stage retraction:
				first $\cp_1$ maps the black diamonds onto the blue crosses on the light blue sphere,
				and second $\cp_2$ maps blue crosses onto the red dots on the red circle by following
				trajectories on the sphere.
				Right: the closest point function $\hat{\cp}(x) = \hat{\cp}_1 \circ \hat{\cp}_2(x)$ first
				maps the black diamonds onto the blue crosses in the light blue plane,
				then maps the blue crosses onto the red dots on the red circle
				by following trajectories contained in the plane.
			}\label{fig:circ1}
\end{figure}

Figure~\ref{fig:circ1} shows this construction of $\cp$ schematically.
The maximal band $B(S)$ around $S$---where $\cp$ is defined---is $\R^3$ without the $x_3$-axis, since the $x_3$-axis
gets retracted by $\cp_1$  to the north-/south-pole of the sphere $S_1$ where $\cp_2$ is not defined.

This first example is intended to highlight the concept of our approach.
In fact, in this particular case, it is much simpler to first project onto $S_2$ by
$\hat{\cp}_2(x) = (x_1,x_2,1/2)^T$, and then retract $S_2$-intrinsically (i.e., in the plane) onto $S$
by $\hat{\cp}_1: S_2 \wo \{0\} \to S$,
to obtain $\hat{\cp}(x) = \hat{\cp}_1 \circ \hat{\cp}_2(x)$
This approach is also illustrated in Figure~\ref{fig:circ1} and, in this particular case,
yields the same closest point function (and in fact they are both equal to $\ecp$).

\subsection{Example 2}\label{subsect:ex2}
We consider a curve embedded in $\R^3$ given as the intersection of a cylinder with a parabola (see Figure~\ref{fig:circle_pringle} right).
The two level set functions are
\begin{equation}\label{eqn:pringle}
	\begin{aligned}
		\varphi_1 (x_1,x_2,x_3) &= 1 - x_2^2 - x_3^2\;, &
		\varphi_2 (x_1,x_2,x_3) &= x_3- x_1^2 \;.
	\end{aligned}
\end{equation}

This time we find the closest point functions numerically using the ODE solver \texttt{ode45} in MATLAB.
The first closest point function $\cp$---which first maps onto the cylinder---is obtained by solving the two ODEs below
\begin{align*}
	\eta_1 &= -\frac{\nabla \varphi_1}{|\nabla \varphi_1|^2} \circ \eta_1 \;, \quad \eta_1(0,x) = x \;, &
	& \Rightarrow \quad z := \eta_1(\varphi_1(x),x) \;, \\
	\eta_2 &= -\frac{P_1 \nabla \varphi_2}{|P_1 \nabla \varphi_2|^2} \circ \eta_2 \;, \quad \eta_2(0,z) = z \;,  &
	& \Rightarrow \quad \cp(x) := \eta_2( \varphi_2(z), z) \;,
\end{align*}
in the given order.
In the same way we obtain the second function $\hat{\cp}$---which maps first onto the parabola---by interchanging the roles of $\varphi_1$ and $\varphi_2$
and numerically solving the ODEs
\begin{align*}
	\eta_2 &= -\frac{\nabla \varphi_2}{|\nabla \varphi_2|^2} \circ \eta_2 \;, \quad \eta_2(0,x) = x \;, &
	& \Rightarrow \quad z := \eta_2(\varphi_1(x),x) \;, \\
	\eta_1 &= -\frac{P_2 \nabla \varphi_1}{|P_2 \nabla \varphi_1|^2} \circ \eta_1 \;, \quad \eta_1(0,y) = z \;, &
	& \Rightarrow \quad  \hat{\cp}(x) := \eta_1( \varphi_1(z), z).
\end{align*}
The retraction stages of the resulting closest point functions $\cp$ and $\hat{\cp}$ are visualized in Figure~\ref{fig:pr1}.

In contrast with our first example, here the two closest point functions $\cp$ and $\hat{\cp}$ are different.
We define a $50 \times 50 \times 50$ Cartesian grid $G$ on a reference box $R =  [-1.25, 1.25] \times [-1.25, 1.25] \times  [-0.25, 1.25]$ containing $S$.
We use a subset of these points as a narrow band of grid points
surrounding $S$.\footnote{For this particular example we define a band by $B(S) = \{ (x_1,x_2,x_3) : \varphi \leq 0.125\}$ where we use
$\varphi = (\varphi_1^2 + \varphi_2^2)^{\frac{1}{2}}$
in lieu of Euclidean distance. Our banded grid is $G \cap B(S)$ and contains 1660 grid points.
It contains more points than are strictly necessary: see \cite[Appendix A]{cbm:icpm} for an approach to banded grids for the Closest Point Method.}
Comparing the values of the distance function $\varphi = (\varphi_1^2 + \varphi_2^2)^{\frac{1}{2}}$
at the closest points
\begin{align*}
	& \max\limits_{x \in  G \cap B(S) } \varphi( \, \cp(x) \,) =  4.4893 \cdot 10^{-15}, \\
	& \max\limits_{x \in  G \cap B(S) } \varphi( \, \hat{\cp}(x) \,) =  4.4758 \cdot 10^{-15}
\end{align*}
we can see that both numerical functions $\cp$ and $\hat{\cp}$ are very accurate.  We also note they are indeed different mappings because $\max_{x \in  G \cap B(S) } |\cp(x) - \hat{\cp}(x)| = 1.7095 \cdot 10^{-03}$
(see also Table~\ref{tab:ErrTabHeat1} where they are clearly distinct from $\ecp$).
\begin{figure}
	\begin{center}
		\includegraphics[width=.43\textwidth]{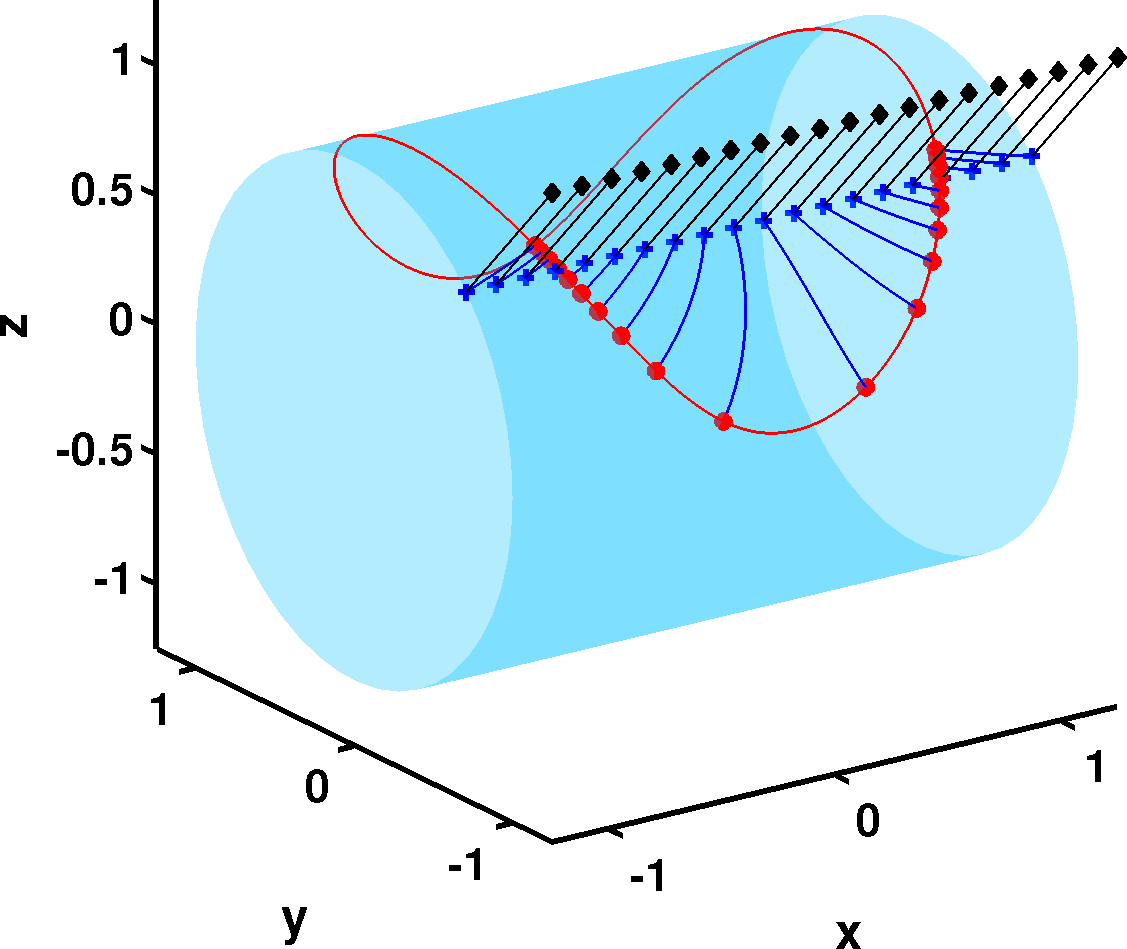}%
                \hspace*{.06\textwidth}%
		\includegraphics[width=.43\textwidth]{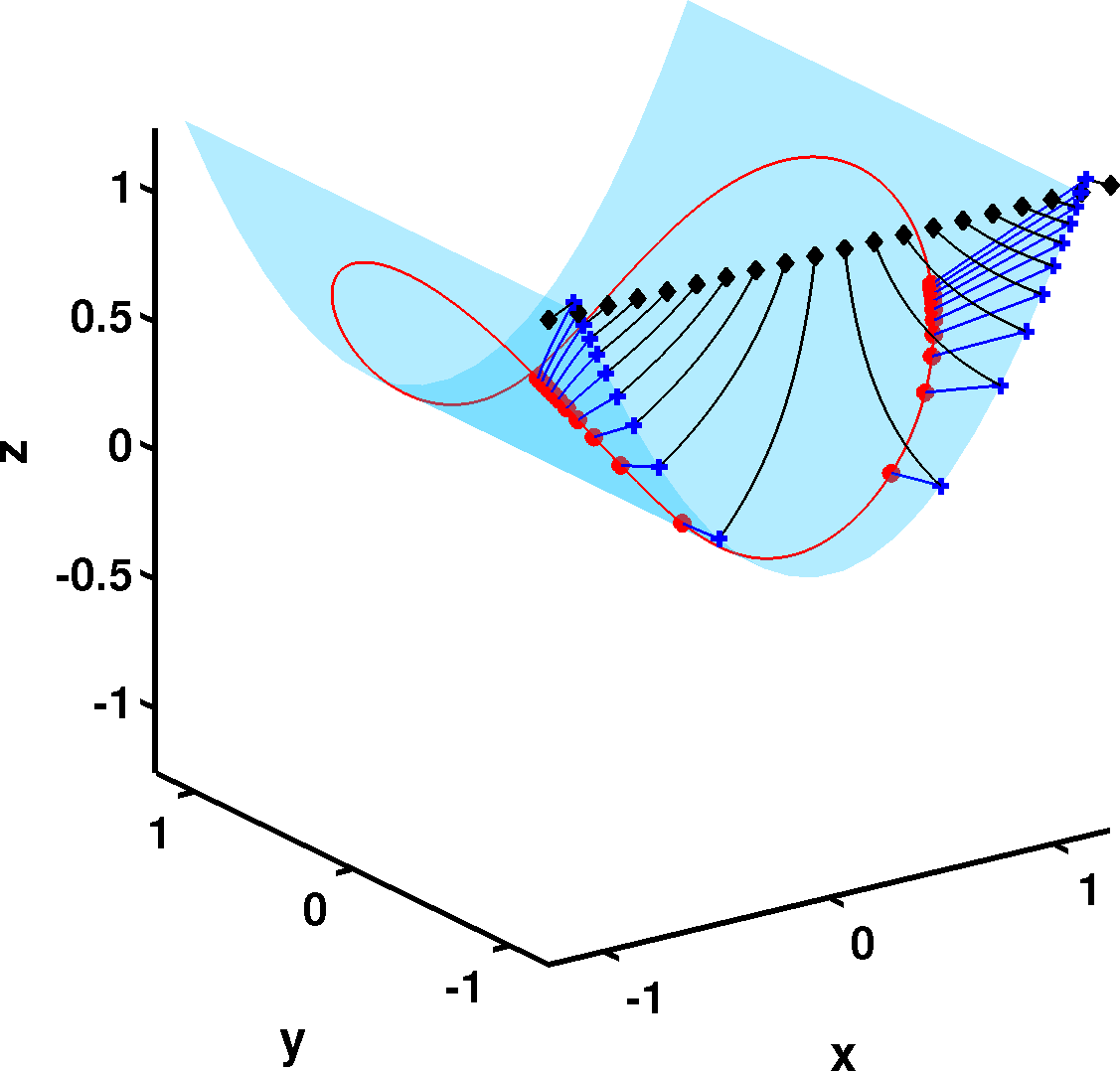}
	\end{center}
	\caption{A curve (red) embedded in $\R^3$, given as illustrated in Figure~\ref{fig:circle_pringle} (right). Left:
				the closest point function $\cp = \cp_2 \circ \cp_1$ maps the black diamonds to the red dots by a two-stage retraction:
				first $\cp_1$ maps the black diamonds onto the blue crosses on the light blue cylinder, and
				second $\cp_2$ maps the blue crosses onto the red dots on the red curve by following
				trajectories on the cylinder surface.
				Right: the closest point function $\hat{\cp} = \hat{\cp}_1 \circ \hat{\cp}_2$
				first maps the black diamonds onto the blue crosses on the light blue parabola by $\hat{\cp}_2$,
				then $\hat{\cp}_1$ maps the blue crosses onto the red dots on the red curve by following
				trajectories on the parabola.}\label{fig:pr1}
\end{figure}

\section{The Closest Point Method with Non-Euclidean Closest Point Functions}\label{sect:CPM}
In this section we demonstrate that the explicit Closest Point Method based on Euler time-stepping
still works when replacing the Euclidean closest point function with non-Euclidean ones.

Given an evolution equation on a smooth closed surface $S$ as
\begin{align*}
	\pd_t u - A_S(t,y,u) &= 0  \;, &  u(0,y) &= u_0(y) \;, \quad y \in S
\end{align*}
where $A_S$ is a linear or nonlinear surface-spatial differential operator on $S$, following \cite{sjr:SimpleEmbed},
the semi-discrete explicit Closest Point Method based on Euler time-stepping with time step $\tau$ is
\begin{align*}
	&\text{Initialization:} & v_0 & = u_0 \circ \cp \\
	&\text{Evolve step:} & w_{n+1} &= v_n + \tau A(t_n,x,v_n) \;, \quad x \in B(S) \\
	&\text{Extension step:} & v_{n+1} &= w_{n+1} \circ \cp
\end{align*}
where $A$ is a spatial operator on $B(S)$ which is defined from $A_S$ via the closest point calculus and hence
\[
	A(t,x,u \circ \cp)|_{x = y} = A_S(t,y,u) \;, \quad y \in S \;.
\]
For the fully discrete Closest Point Method which is also discrete with respect to the spatial variable $x \in B(S)$
we replace $A$ with a discretization of it, while the extension $w_{n+1} \circ \cp(x)$ is replaced with interpolation
of the discrete $w_{n+1}$ in a neighborhood of $\cp(x)$ because the point $\cp(x)$ is not a grid point in general \cite{sjr:SimpleEmbed}.

\subsection{Advection Equation}
We consider an advection problem with a constant unit speed on the curve $S$ shown in Figure~\ref{fig:circle_pringle} (right):
\begin{equation}\label{eqn:Advec}
	\begin{aligned}
		\pd_t u + \diver_S \left( u \cdot T(y) \right) &= 0 \;,  &
		u(0,y) &= y_3 \;, \quad y \in S \;, \\
		&& \text{with} \quad T(y) &= - \frac{\nabla \varphi_1(y) \times \nabla \varphi_2(y)}{|\nabla \varphi_1(y) \times \nabla \varphi_2(y)|} \;.
	\end{aligned}
\end{equation}
The surface-spatial operator is $A_S(y,u) = -\diver_S \left( u \cdot T(y) \right)$, while the operator used in the Closest Point Method is
\[
	A(x,v) = -\diver(v \cdot T \circ \cp(x)) \;, \quad v = u \circ \cp \;, \quad x \in B(S)
\]
according to the Divergence Principle~\ref{cor:PrinDiv}.

We solved this advection problem with three different closest point functions---namely $\cp$, $\hat{\cp}$ from Example~2, and
the Euclidean closest point function $\ecp$---and with three different mesh sizes $h$ on the reference box $R =  [-1.25, 1.25] \times [-1.25, 1.25] \times  [-0.25, 1.25]$.
The $\ecp$ was computed with a numerical optimization procedure using Newton's method.
The evolve step uses the first order accurate Lax--Friedrichs scheme (time step $\tau = 0.95 h$ in accordance with the CFL-condition) and the extension step is performed with
WENO interpolation \cite{cbm:lscpm} (based on tri-quadratic interpolation).
The solution is advected until time $t = 1$.

We compare to a highly accurate solution obtained from parametrizing the problem.
By using the parametrization $\gamma : [0, 2 \pi) \to S$, $\gamma(\theta) = (\cos(\theta), \sin(\theta) \sqrt{1+\cos(\theta)^2}, \cos(\theta)^2)^T$
the advection problem \eqref{eqn:Advec} is equivalent to
\begin{equation}\label{eqn:AdvecParam}
  |\gamma'(\theta)| \pd_t \bar{u} - \pd_\theta \bar{u} = 0 \;,
  \qquad\qquad
  \bar{u}(0,\theta) = \cos(\theta)^2 \;,
  \quad
  \bar{u}(t,2\pi) = \bar{u}(t,0) \;,
\end{equation}
where $\bar{u}(t,\theta) = u(t,\gamma(\theta))$. The formal solution of the latter is
\begin{equation}\label{eqn:AdvecSol}
	\bar{u}(t,\theta) = \cos \left( s^{-1}\left( t + s(\theta) \right) \right)^2 \;, \quad \text{where} \quad s(\theta) = \int\limits_0^\theta |\gamma'(\alpha)| \; d\alpha
\end{equation}
is the arc-length function. We used Chebfun \cite{chebfunv4.1} within MATLAB
to find a highly accurate approximation to $s(\theta)$ and thereafter a Newton iteration to approximate the value $s^{-1}\left( t + s(\theta) \right)$.

Table~\ref{tab:ErrTabAdvec} shows the error in the results measured in $l_{\infty}$-norm at $t=1$.
We see there is no significant difference regarding the choice of the closest point function
and that the error is of order $O(h)$ as expected from the Lax--Friedrichs scheme.

\begin{table}
  \caption{Errors measured in $l_{\infty}$-norm at stop time $t=1$ for Closest Point Method approximation of the
    advection problem \eqref{eqn:Advec} on the curve shown in Figure~\ref{fig:circle_pringle} (right)
	 using various closest point functions.
    There is no significant difference regarding the choice of the
    closest point function.}
  \label{tab:ErrTabAdvec}
	\begin{center}
		\begin{tabular}{|c|c|c|c|}
			\hline
			$h$ & $\cp$ & $\hat{\cp}$ & $\ecp$ \\
			\hline
			0.0125 &  9.1630e-03 &  9.1592e-03 &  9.0615e-03 \\
			0.00625 &  4.6182e-03 &  4.6172e-03 &  4.5544e-03 \\
			0.003125 &  2.3262e-03 &  2.3260e-03 &  2.2908e-03 \\
			\hline
		\end{tabular}
	\end{center}
\end{table}

\subsection{Diffusion Equation}
Here, we consider the diffusion equation on the curve $S$ shown in Figure~\ref{fig:circle_pringle} (right):
\begin{equation}\label{eqn:Heat}
	\begin{aligned}
		\pd_t u - \laplace_S u &= 0 \;, &&&
		u(0,y) &= \frac{\exp(4 y_3)}{50} \;, \quad y \in S \;.
	\end{aligned}
\end{equation}
The surface-spatial operator is now $A_S(y,u) = \laplace_S u$.

Here again, we compare to a highly accurate solution obtained from parametrizing the problem.
By $\bar{u}(t,\theta) = u(t,\gamma(\theta))$, where $\gamma$ is the same parametrization that we used in \eqref{eqn:AdvecParam},
the diffusion equation~\eqref{eqn:Heat} transforms to
\begin{equation*}
	\pd_t \bar{u} - \frac{1}{|\gamma'(\theta)|} \pd_\theta \left( \frac{1}{|\gamma'(\theta)|} \pd_\theta \bar{u} \right) = 0 \;,\quad
	\bar{u}(0,\theta) = \frac{\exp(4 \cos(\theta)^2)}{50} \;,\quad \bar{u}(t,2\pi) = \bar{u}(t,0) \;.
\end{equation*}
The formal solution of the latter (with frequency parameter $\omega = 2\pi/s(2\pi)$ and arc-length function $s(\theta)$) is given by
\begin{align*}
	\bar{u}(t,\theta) &= \sum\limits_{m=-\infty}^{\infty} c_m \; e^{-\omega^2 m^2 t} \; e^{i \omega m \, s(\theta)} \;, &
	c_m & = \frac{1}{s(2\pi)} \int\limits_{0}^{2\pi}  \frac{\exp(4 \cos(\theta)^2)}{50} \; e^{-i \omega m \, s(\theta)} \; |\gamma'(\theta)| \; d\theta \;.
\end{align*}
We again used Chebfun \cite{chebfunv4.1} to obtain highly accurate approximations to $c_m$ and $s(\theta)$. Moreover, we restrict the summation to $-M \leq m \leq M$ where $M$
is chosen such that the bound $e^{-\omega^2 M^2 t} c_0$  is lower than machine accuracy.

In Section~\ref{sect:diffusion}, we have discussed three different ways to deal with the Laplace--Beltrami operator by the closest point calculus which in turn give rise to three different operators $A$.
Now, we solve the diffusion equation \eqref{eqn:Heat} using the Closest Point Method with each of these (in the evolve step) and
with each of our three different closest point functions $\ecp$, and $\cp$, $\hat{\cp}$ from Example 2.  The extension step is performed with tri-cubic interpolation.
The time step is chosen as $\tau = 0.2 h^2$ for numerical stability.
The errors are measured at time $t=0.1$ in the $l_{\infty}$-norm.

\begin{enumerate}[1.]
	\item Writing the Laplace--Beltrami operator as $\laplace_S u = \diver_S ( \nabla_S u)$, the simplest possibility is
			\begin{equation}\label{eqn:lap1}
				A(v) = \diver (\nabla v) = \laplace v \;.
			\end{equation}
			But, recall that this need not work if the closest point function does not satisfy the requirement of Theorem~\ref{theo:tensor2}.
			$A$ is discretized by the usual $\Order(h^2)$-accurate $7$-point stencil.
			 Table~\ref{tab:ErrTabHeat1} shows the errors.
			 We observe convergence at the expected rate $\Order(h^2)$ when using the Euclidean closest point function $\ecp$,
			 while $\cp$ and $\hat{\cp}$ do not yield a convergent method.
			 With $\ecp$, as opposed to $\cp$ and $\hat{\cp}$, we can be sure that $A|_S = A_S$ since $\ecp$ satisfies
			 the requirement of Theorem~\ref{theo:tensor2}.

			\begin{table}
				\caption{Errors measured in $l_{\infty}$-norm at stop time $t=0.1$ for Closest Point Method approximation of the
							diffusion equation~\eqref{eqn:Heat} on the curve shown in Figure~\ref{fig:circle_pringle} (right)
							using various closest point functions.
							The spatial operator on the embedding space is $A(v) = \laplace v$ as in \eqref{eqn:lap1}.
							We observe convergence when using $\ecp$ as expected. But, here, $\cp$ and $\hat{\cp}$ do not yield a convergent method
							as they (with hindsight) do not satisfy the requirement of Theorem \ref{theo:tensor2}.}
				\label{tab:ErrTabHeat1}
				\begin{center}
					 \begin{tabular}{|c|c|c|c|}
						\hline
						$h$ & $\cp$ & $\hat{\cp}$ & $\ecp$ \\
						\hline
						0.0125 &  0.017394 &  0.017394 &  4.717879e-04 \\
						0.00625 &  0.017499 &  0.017499 &  1.173944e-04 \\
						0.003125 &  0.017526 &  0.017526 &  2.927441e-05 \\
						\hline
					 \end{tabular}
				\end{center}
		  \end{table}
	\item Now, we rewrite the Laplace--Beltrami operator as $\laplace_S u = \diver_S ( P \nabla_S u)$, where the projector $P$ is given by the outer product $P(y) = T(y) \cdot T(y)^T$
			and $T$ is the tangent field as in \eqref{eqn:Advec}. Based on this formulation we obtain the following operator
			\begin{equation}\label{eqn:lap2}
				A(x,v) = \diver ( P \circ \cp(x) \cdot \nabla v ) = \sum\limits_{l=1}^3 \pd_{x_l} \left( \sum\limits_{m=1}^3 P_{lm} \circ \cp(x) \; \pd_{x_m} v \right) \;,
			\end{equation}
			and Theorem \ref{theo:tensor} ensures that the operators coincide on the surface, i.e., $A|_S = A_S$ for all
			closest point functions in accordance with Definition \ref{def:cpop}.
			$A$ is discretized by replacing the partial differential operators $\pd_{x_l}$ with corresponding $\Order(h^2)$-accurate central difference operators.
			Table~\ref{tab:ErrTabHeat2} shows the errors.
			We observe convergence at the expected rate $\Order(h^2)$ for all three closest point functions and there is no significant difference regarding the choice of the closest point function.
			Note that this approach does require that we know the tangent field.

			\begin{table}
				\caption{Errors measured in $l_{\infty}$-norm at stop time $t=0.1$ for Closest Point Method approximation of the
				  diffusion equation~\eqref{eqn:Heat} on the curve shown in Figure~\ref{fig:circle_pringle} (right)
				  using various closest point functions.
				  The spatial operator on the embedding space is $A(x,v) = \diver ( P \circ \cp(x) \cdot \nabla v )$ as in \eqref{eqn:lap2}.
				  There is no significant difference regarding the choice of the
				  closest point function.}
				\label{tab:ErrTabHeat2}
				\begin{center}
					 \begin{tabular}{|c|c|c|c|}
						\hline
						$h$ & $\cp$ & $\hat{\cp}$ & $\ecp$ \\
						\hline
						0.0125 &  4.805473e-05 &  4.856909e-05 &  4.793135e-05 \\
						0.00625 &  1.198697e-05 &  1.211579e-05 &  1.195396e-05 \\
						0.003125 &  2.975015e-06 &  3.008249e-06 &  2.968171e-06 \\
						\hline
					 \end{tabular}
				\end{center}
		  \end{table}
	\item Finally, we work with re-extensions and appeal to the Gradient and Divergence Principles~\ref{cor:PrinGrad} and~\ref{cor:PrinDiv}:
			\begin{equation}\label{eqn:lap3}
				A(v) = \diver( \nabla v \circ \cp) = \sum\limits_{l=1}^3 \pd_{x_l} [(\pd_{x_l} v) \circ \cp ] \;.
			\end{equation}

			The discretization of $A$ involves two steps: the partial differential operators $\pd_{x_l}$ are replaced with corresponding central difference operators,
			the re-extensions $(\pd_{x_l} v) \circ \cp$ are replaced with tri-cubic interpolation.
			Table~\ref{tab:ErrTabHeat3} shows the errors.
			We observe convergence at the expected rate $\Order(h^2)$ for all three closest point functions.
			There is no significant difference regarding the choice of the closest point function.

			\begin{table}
				\caption{Errors measured in $l_{\infty}$-norm at stop time $t=0.1$ for Closest Point Method approximation of the
				  diffusion equation~\eqref{eqn:Heat} on the curve shown in Figure~\ref{fig:circle_pringle} (right)
				  using various closest point functions.
				  The spatial operator on the embedding space is $A(v) = \diver( \nabla v \circ \cp)$ as in \eqref{eqn:lap3}.
				  There is no significant difference regarding the choice of the
				  closest point function.}
				\label{tab:ErrTabHeat3}
				\begin{center}
					 \begin{tabular}{|c|c|c|c|}
						\hline
						$h$ & $\cp$ & $\hat{\cp}$ & $\ecp$ \\
						\hline
						0.0125 &  1.508222e-04 & 1.503411e-04  &  1.492513e-04  \\
						0.00625 &  3.754836e-05 & 3.743075e-05  &  3.715598e-05 \\
						0.003125 &  9.347565e-06 & 9.319686e-06  &  9.251365e-06 \\
						\hline
					 \end{tabular}
				\end{center}
		  \end{table}
\end{enumerate}
\smallskip

In the three experiments above we solved the same surface problem \eqref{eqn:Heat}.
In experiment~1 we confirm that when using a special closest point function satisfying the requirements of
Theorem~\ref{theo:tensor2}---namely $\ecp$---we can drop the second extension. This reduces the computational costs (fewer interpolation operations) compared to experiment~3 where
we explicitly use the re-extension. But comparing the right-most column of Table~\ref{tab:ErrTabHeat1} with Tables~\ref{tab:ErrTabHeat2} and \ref{tab:ErrTabHeat3} we can see
that either using more knowledge about the surface (the explicit use of $P$) as in experiment~2 or doing more work (re-extensions) as in experiment~3 pays off---at least in this particular problem---with
better error constants.

Moreover, experiment~3 (as well as the advection experiment above)
provides some evidence that when using re-extensions after each
differential operator (that is, using the framework of
Section~\ref{sect:calcCP}) all closest point functions work equally
well, both in theory and in practice.
The choice of the closest point function in practice might depend, for
example, on which closest point function is easiest to construct.

\section{Conclusions}
We presented a closest point calculus which makes use of closest point functions.
This calculus forms the basis of a general Closest Point Method along the lines of \cite{sjr:SimpleEmbed}.
The original Closest Point Method of \cite{sjr:SimpleEmbed} inspired the present work:
after becoming aware of the key property $D \ecp(y) = P(y)$ which accounts for the gradient principle of the original method,
we turned it into a definition of a general class of closest point functions namely Definition~\ref{def:cpop}.
Here, we characterize such closest point functions (see Theorems \ref{theo:cpopDiffInv} and \ref{theo:CPbyGeo}) and prove that this class of closest point functions yields the desired
closest point calculus (see Theorem~\ref{theo:cpopDiff} and the Gradient and Divergence Principles~\ref{cor:PrinGrad} and~\ref{cor:PrinDiv} which follow).
The closest point calculus is also sufficient to tackle higher order surface differential operators by appropriate combinations of the principles,
however for all surface-intrinsic diffusion operators the calculus can be simplified
to use fewer closest point extensions (see Theorems~\ref{theo:tensor} and~\ref{theo:tensor2}).
Finally, the basic principles of the original Closest Point Method are now contained and proven in our general framework.

In addition to the framework, we describe a construction of closest point functions given a level-set description of the surface and show examples.
This demonstrates that there are interesting closest point functions besides the Euclidean one. Furthermore, we demonstrate on two examples (the advection equation and the diffusion equation on
a closed space curve) that the Closest Point Method combined with non-Euclidean closest point functions works as expected.

\bibliographystyle{siam}
\bibliography{CPFunctions}{}

\end{document}